\documentclass[12pt,reqno]{amsart}
\usepackage[a4paper]{geometry}
\usepackage{amssymb,amsmath,amsthm,enumerate,verbatim,bbm}
\usepackage[mathscr]{euscript}

\usepackage{xcolor}
\usepackage{mathrsfs}

\DeclareMathOperator{\Ran}{Ran}

\DeclareMathOperator{\Ker}{Ker}

\DeclareMathOperator{\rank}{rank}

\DeclareMathOperator{\Clos}{Clos}
\DeclareMathOperator{\Span}{span}

\DeclareMathOperator{\BMOA}{BMOA}

\renewcommand\Im{\hbox{{\rm Im}}\,}

\newcommand{\abs}[1]{{\lvert#1\rvert}}

\newcommand{\norm}[1]{\lVert#1\rVert}

\newcommand{\bbT}{{\mathbb T}}
\newcommand{\bbR}{{\mathbb R}}
\newcommand{\bbC}{{\mathbb C}}

\newcommand{\bbZ}{{\mathbb Z}}
\newcommand{\bbD}{{\mathbb D}}

\newcommand{\calH}{{\mathcal H}}

\newcommand{\calK}{\mathcal{K}}
\newcommand{\bcalK}{\pmb{\mathcal{K}}}

\newcommand{\calC}{\mathcal{C}}

\newcommand{\bbf}{\mathbf f}

\newcommand{\bu}{\mathbf u}

\newcommand{\bC}{\mathbf C}
\newcommand{\bH}{\mathbf H}

\newcommand{\bpsi}{\pmb \psi}
\newcommand{\bP}{\mathbf P}
\newcommand{\bT}{\mathbf T}
\newcommand{\bp}{\mathbf p}
\newcommand{\btheta}{{\pmb \theta}}
\newcommand{\bvarphi}{{\pmb \varphi}}

\numberwithin{equation}{section}
\renewcommand{\[}{\begin{equation}}
\renewcommand{\]}{\end{equation}}

\theoremstyle{plain}
\newtheorem{theorem}{\bf Theorem}[section]
\newtheorem*{theorem*}{Theorem 1.1$'$}
\newtheorem{lemma}[theorem]{\bf Lemma}
\newtheorem{proposition}[theorem]{\bf Proposition}
\newtheorem{corollary}[theorem]{\bf Corollary}

\theoremstyle{definition}

\newtheorem*{definition*}{\bf Definition}

\theoremstyle{remark}
\newtheorem*{remark*}{\bf Remark}
\newtheorem{remark}[theorem]{\bf Remark}

\DeclareFontFamily{U}{mathx}{\hyphenchar\font45}
\DeclareFontShape{U}{mathx}{m}{n}{<5> <6> <7> <8> <9> <10>
<10.95> <12> <14.4> <17.28> <20.74> <24.88> mathx10}{}
\DeclareSymbolFont{mathx}{U}{mathx}{m}{n}
\DeclareFontSubstitution{U}{mathx}{m}{n}
\DeclareMathAccent{\wc}{0}{mathx}{"71}

\newcommand{\wt}{\widetilde}
\newcommand{\wh}{\widehat}

\newcommand{\1}{\mathbbm{1}}
\newcommand{\proj}{P}
\newcommand{\real}{{\rm real}}

\begin{document}

\title[Schmidt subspaces for Hankel operators]{Weighted model spaces and Schmidt subspaces of Hankel operators}

\author{Patrick G\'erard}
\address{Universit\'e Paris-Sud XI, Laboratoire de Math\'ematiques d'Orsay, CNRS, UMR 8628, et Institut Universitaire de France}
\email{Patrick.Gerard@math.u-psud.fr}

\author{Alexander Pushnitski}
\address{Department of Mathematics, King's College London, Strand, London, WC2R~2LS, U.K.}
\email{alexander.pushnitski@kcl.ac.uk}



\begin{abstract}For a bounded Hankel matrix $\Gamma$, we describe the structure of the Schmidt subspaces 
of $\Gamma$, namely the eigenspaces of $\Gamma^* \Gamma$ corresponding to non zero eigenvalues. 
We prove that these subspaces are in correspondence with weighted model spaces in the Hardy space on the unit circle. 
Here we use the term ``weighted model space" to describe the range of an isometric multiplier acting on a model space. 
Further, we obtain similar results for Hankel operators acting in the Hardy space on the real line. 
Finally, we give a streamlined proof of the Adamyan-Arov-Krein theorem  using the language of weighted model spaces. 
\end{abstract}

\date{10 July 2019}

\maketitle

\section{Introduction and main results}\label{sec.a}

\subsection{Overview}

Let $\Gamma=\{\gamma_{j+k}\}_{j,k=0}^\infty$ be a Hankel matrix, which we assume
to be bounded (but not necessarily compact) on $\ell^2=\ell^2(\bbZ_+)$. 
Under the standard identification between $\ell^2$ and the Hardy space $\calH^2(\bbT)$
(precise definitions are given below), 
$\Gamma$ can be considered as a bounded operator on $\calH^2(\bbT)$. 
It is well known that the kernel of $\Gamma$ is an invariant subspace of the shift operator, and
therefore by Beurling's theorem \cite{Beurling} it is either trivial or can be identified with a subspace of the form $\psi\calH^2(\bbT)$
for some inner function $\psi$. 

The question we address in this paper is: 
\emph{How can one characterise eigenspaces $\Ker(\Gamma^*\Gamma-s^2I)$, $s>0$, as a class of subspaces in the Hardy space?}

Our main result (Theorem~\ref{thm.z0}) is that every such eigenspace can be identified with a subspace of the form
\[
\Ker(\Gamma^*\Gamma-s^2I)=
p\calK_{z\theta}\subset\calH^2(\bbT),
\label{zz}
\]
where $\theta$ is an inner function, 
$\calK_{z\theta}=\calH^2\cap (z\theta\calH^2)^\perp$ is 
a model space, and $p$ is an \emph{isometric multiplier} on $\calK_{z\theta}$. 
Furthermore, we show that the action 
\[
\Gamma: \Ker (\Gamma^*\Gamma-s^2 I)\to \Ker (\Gamma\Gamma^*-s^2 I)
\label{z00}
\]
is given by a simple explicit formula, which is completely 
determined by $s$, $p$ and $\theta$. 
Finally, we prove that the degree of the inner part of $p$ coincides with the total multiplicity of the
spectrum of $\Gamma^*\Gamma$ in the interval $(s^2,\infty)$. 
As a simple corollary, we also characterise all eigenspaces of
self-adjoint Hankel matrices $\Gamma$. 

Important precursors to our results are the Adamyan-Arov-Krein (AAK) theory \cite{AAK3}
and the description \cite{GGAst} of eigenspaces of $\Gamma^*\Gamma$ for \emph{compact}
Hankel operators $\Gamma$ (in which case the corresponding inner functions $\theta$ reduce to finite Blaschke products). 
However, the isometric multiplier structure \eqref{zz} seems to be a new result even for finite rank Hankel operators. 

We were also inspired by works on the structure of \emph{nearly $S^*$-invariant subspaces}
and the kernels of Toeplitz operators; see \cite{Hayashi,Hitt,Sarason} for early works  and
\cite{Aleman,Dyakonov,Partington,Fricain,Hartmann} for later developments and surveys of the subject. 
In a companion paper \cite{paperB} we consider in detail the interesting special case when 
the spectrum of $\Gamma^*\Gamma$ is finite, solve an inverse spectral problem involving the parameters $s$, $\theta$
and give an explicit description of the corresponding class of symbols.

We shall mention here one important aspect of our proof: it turns out that the natural approach to this problem 
is to consider $\Gamma$ 
together with the ``shifted" Hankel matrix
$\wt \Gamma=\{\gamma_{j+k+1}\}_{j,k=0}^\infty$.
Another feature of our approach is the analysis of a certain new class of subspaces of the Hardy space, 
which generalise nearly $S^*$-invariant subspaces.

Let us recall some terminology. 
A real number $s>0$ is called a \emph{singular value} of $\Gamma$, if 
$s^2$ is an eigenvalue of the operator $\Gamma^*\Gamma$. 
This is slightly wider than the standard definition: we assume neither compactness
of $\Gamma$ nor $s^2$ to be an isolated point in the spectrum of $\Gamma^*\Gamma$.
We note that for a bounded Hankel operator $\Gamma$, the operator $\Gamma^*\Gamma|_{(\Ker \Gamma)^\perp}$ can have 
arbitrary spectrum, see \cite{Treil}.

If $s$ is a singular value of $\Gamma$, then a pair $\{\xi,\eta\}$ of elements of $\ell^2$ is called
a \emph{Schmidt pair} (more precisely, an $s$-Schmidt pair) of $\Gamma$, if it satisfies
$$
\Gamma \xi=s\eta, \quad \Gamma^*\eta=s\xi.
$$
Clearly, $s$-Schmidt pairs form a linear subspace of dimension $\dim\Ker(\Gamma^*\Gamma-s^2I)\leq \infty$.
We will call $\Ker(\Gamma^*\Gamma-s^2I)$ the \emph{Schmidt subspace} of $\Gamma$. 
The problem of description of all $s$-Schmidt pairs of $\Gamma$ is equivalent to the problem of the 
description of the action \eqref{z00}.

We briefly describe the structure of this rather long introductory section. 
In Sections~\ref{sec.z1} and \ref{sec.z2} we describe the realisation of Hankel matrices on the Hardy space. 
In Section~\ref{sec.z3} for the purposes of comparison we recall the classical Adamyan-Arov-Krein theorem. 
In Section~\ref{sec.z4} as a warm-up we consider Hankel operators with inner symbols; this allows us to 
introduce model spaces into the subject in a natural way. 
In Section~\ref{sec.z5} we discuss isometric multipliers on model spaces. 
In Section~\ref{sec.z6} we state and discuss our main result. 
In Section~\ref{sec.z7} we consider the special case of self-adjoint Hankel matrices $\Gamma$
and describe their eigenspaces. 
In Section~\ref{sec.z8} we state the analogue of our main result for Hankel operators acting
on the Hardy space $\calH^2(\bbR)$ of the real line.

\subsection{Anti-linear operators}\label{sec.z1}
Let us denote by $\calC$ the anti-linear operator of complex conjugation on $\ell^2$: 
$\calC \xi=\overline{\xi}$. Observe that the Hankel matrix $\Gamma$ is symmetric and therefore
$\Gamma\calC =\calC\Gamma^*$. 
It follows that 
$$
\xi\in \Ker(\Gamma^*\Gamma-s^2I)\quad\Leftrightarrow\quad \overline{\xi}\in\Ker(\Gamma\Gamma^*-s^2I)
$$
and the \emph{anti-linear} map $\Gamma\calC$ maps $\Ker(\Gamma\Gamma^*-s^2I)$ onto itself.
Further, it is clear  that the map
\[
s^{-1}\Gamma\calC: \Ker(\Gamma\Gamma^*-s^2I)\to \Ker(\Gamma\Gamma^*-s^2I) 
\label{z3}
\]
is an involution. 

Thus, the problem of the description of the action \eqref{z00} of $\Gamma$ is 
equivalent to the problem of the description of the involution \eqref{z3}. 
As we shall see, the point of view \eqref{z3} offers some advantages. 
Observe that $(\Gamma\calC)^2$ is a \emph{linear} operator: 
$$
(\Gamma\calC)^2=\Gamma \calC \Gamma\calC =\Gamma\Gamma^*.
$$
Thus, we can rephrase our aim as follows: 
we describe the eigenspaces 
$\Ker ((\Gamma\calC)^2-s^2I)$ and the anti-linear involution $s^{-1}\Gamma\calC$ on these eigenspaces.

\subsection{Mapping onto the Hardy space}\label{sec.z2}
Let $\calH^2(\bbT)\subset L^2(\bbT)$ be the standard Hardy space, and 
let $\proj:L^2(\bbT)\to \calH^2(\bbT)$ be the corresponding orthogonal projection (the Szeg\H o projection). 
For $f\in \calH^2(\bbT)$, let $\wh f_j$  be the $j$'th Fourier coefficient:
$$
\wh f_j=\int_0^{2\pi} f(e^{i\theta})e^{-ij\theta}\frac{d\theta}{2\pi}. 
$$
Conversely, for a sequence $\xi\in\ell^2$, we denote by $\wc \xi\in \calH^2(\bbT)$ the function 
$$
\wc \xi(z)=\sum_{j=0}^\infty \xi_j z^j, \quad \abs{z}\leq1.
$$

It is standard to study Hankel operators on the Hardy space, i.e. to consider
an operator on the Hardy space whose matrix with respect to the standard basis
is given by $\{\gamma_{j+k}\}$. 
In the same way, we map the anti-linear operator $\Gamma\calC$ to the Hardy space as follows. 

For a \emph{symbol} $u\in \calH^2(\bbT)$, we consider the anti-linear Hankel operator 
$H_u$, formally defined  by 
\[
H_u f=\proj(u\cdot \overline{f}), \quad f\in \calH^2(\bbT).
\label{z0}
\]
We are only interested in bounded Hankel operators; it is well known that this corresponds
to the symbol $u$ being in the $\BMOA(\bbT)$ class (see e.g. \cite[Theorem 1.2]{Peller}). 
We recall that the shift operator $S$ on $\calH^2(\bbT)$ is defined as
$$
Sf(z)=zf(z),\quad  z\in \bbT\ ,
$$
and that the anti-linear Hankel operators $H_u$ are characterised by the identity 
$$
H_uS=S^*H_u \ . 
$$
We denote by $\1$ the function in $\calH^2(\bbT)$ identically equal to $1$; 
obviously, $u=H_u\1$.

Clearly, the operators $\Gamma\calC$ and $H_u$ with $u=\wc\gamma$ are unitarily 
equivalent by the Fourier transform: 
$$
\wc{\Gamma\calC \xi}=H_{u}\wc \xi, \quad u=\wc \gamma.
$$
In particular, $H_u^2$ is a linear self-adjoint operator on $\calH^2(\bbT)$ which is 
unitarily equivalent to $\Gamma\Gamma^*$. 
Coming back to the description of the Schmidt pairs for $\Gamma$, we can 
summarise the above discussion as follows:
\begin{proposition}\label{prp.a1}
A pair $\{\overline{\xi},\eta\}$ is an $s$-Schmidt pair of $\Gamma=\{\gamma_{j+k}\}_{j,k=0}^\infty$ 
if and only if 
$$
H_{u} \wc \xi =s\wc \eta, \quad \wc\xi \in \Ker (H_u^2-s^2I), \quad u=\wc\gamma. 
$$
\end{proposition}
Thus, our aim is to describe the eigenspaces 
$$
E_{H_u}(s):=\Ker(H_u^2-s^2I)\subset \calH^2(\bbT), \quad s>0,
$$
and the action of $H_u$ on these eigenspaces. 
For most of the paper, the symbol $u$ is fixed and so we write $E_{H}(s)$ in place of $E_{H_u}(s)$
when there is no danger of confusion.

In Appendix~\ref{sec.app}, we also discuss the realisation of  the linear operator
$\Gamma$ (rather than the anti-linear operator $\Gamma\calC$) in the Hardy space.

\subsection{The Adamyan-Arov-Krein theorem}\label{sec.z3}

We recall the classical Adamyan-Arov-Krein (AAK) theorem. 
Below $\{s_k(\Gamma)\}_{k=1}^\infty$ is  the ordered sequence of approximation numbers of $\Gamma$, i.e. 
$$
s_k(\Gamma)=\inf\{\norm{\Gamma-F}: \rank F\leq k\},
$$
where the infimum is taken over all linear operators of rank less than or equal to $k$. 
Further, $s_\infty(\Gamma)=\lim_{k\to\infty}s_k(\Gamma)$ is the essential norm of $\Gamma$. 
It is well-known that if $s=s_k(\Gamma)>s_\infty(\Gamma)$, then $s$ is a singular value of $\Gamma$, i.e.
$\Ker(\Gamma^*\Gamma-s^2I)\not=\{0\}$; furthermore, in this case this space is finite-dimensional.

Similarly, for a bounded anti-linear Hankel operator $H_u$, we denote 
$$
s_k(H_u)=\inf\{\norm{H_u-F}: \rank F\leq k, \text{$F$ anti-linear}\}, 
$$
where the infimum is taken over all anti-linear operators of rank less than or equal to $k$, and set 
$s_\infty(H_u)=\lim_{k\to\infty}s_k(H_u)$.

Adapting the notation of the AAK theorem to our setting (see Proposition~\ref{prp.a1}), we quote it as follows. 

\begin{theorem}\cite[Theorem 1.2]{AAK3}
Let $s>0$ be a singular value of a bounded Hankel operator $H_u$. 
Then the functions $f\in E_{H}(s)$, $g=H_u(f)/s$ 
can be represented as
\[
f(z)=I_+(z)I(z)O(z), \quad 
g(z)=I_-(z)I(z)O(z), \quad 
\label{aak1}
\]
where $O$ is an outer function, and $I$, $I_+$ and $I_-$ are inner functions. 
The inner function $I$ is independent of the choice of
$f\in E_H(s)$.
Further, we have
\begin{enumerate}[\rm (i)]
\item
$\deg I_++\deg I_-\leq n-1$, where $n=\dim E_H(s)$.
If $n$ is finite, then all pairs $(f,g)$ can be described by the formula 
\[
f(z)=P(z)I(z)\Omega(z), \quad
g(z)=z^{n-1}\overline{P(z)}I(z)\Omega(z), \quad \abs{z}=1,
\label{aak2}
\]
where  $P$ is an arbitrary polynomial of degree less than or equal to $n-1$,
and $\Omega$ is an outer function which is uniquely defined by $H_u$ 
if we require that $\Omega(0)>0$.
\item
If $s=s_k(H_u)$ and 
$$
s_{k-1}(H_u)>s_k(H_u)\geq s_\infty(H_u),
$$ 
then the degree of $I$ is $k$. Further, in this case there exists a rational function $r$ with 
no poles in $\overline{\bbD}$ such that the rank of $H_r$ is $k$ and 
$$
\norm{H_u-H_r}=s.
$$
\end{enumerate}
\end{theorem}

\subsection{Inner symbols and model spaces}\label{sec.z4}

We recall that for an inner function $\theta$ on the unit disk, the model space $\calK_\theta$ 
is defined by 
$$
\calK_\theta=\calH^2(\bbT)\cap(\theta\calH^2(\bbT))^\perp. 
$$
By Beurling's theorem, any proper $S^*$-invariant subspace of $\calH^2(\bbT)$ 
is a model space.

The basic building blocks of our construction are the Hankel operators $H_u$ whose
symbols are inner functions. 
In this case, as is well known (and easy to prove, see e.g. \cite[Theorem 1.2.6]{Peller}), 
$H_u^2$ is an orthogonal projection whose range is a model space. 
We state this precisely for clarity:  
\begin{lemma}\label{lma.a2}
Let $\theta$ be an inner function. Then $H_\theta^2$ is an orthogonal projection in 
$\calH^2(\bbT)$, with $\Ran H_\theta^2=\Ran H_\theta=\calK_{z\theta}$ described by 
$$
f\in \Ran H_\theta \Leftrightarrow f\in \calH^2(\bbT) \text{ and } \theta\overline{f}\in \calH^2(\bbT). 
$$
Further, $H_\theta$ acts on $\Ran H_\theta$ as an anti-linear involution,
$$
H_\theta f=\theta\overline{f}, \quad f\in \Ran H_\theta. 
$$
\end{lemma}

Observe that we have $\theta,\1\in \Ran H_\theta$ for any inner function $\theta$. 

To illustrate Lemma~\ref{lma.a2}, let us discuss the case of inner functions of finite degree 
(the degree of $\theta$ is  defined as $\dim\calK_{\theta}$). 
An inner function $\theta$ has a finite degree $k$ if and only if it is given by a finite Blaschke product
\[
\theta(z)=e^{i\alpha}\prod_{j=1}^k\frac{z-z_j}{1-\overline{z_j}z}, \quad \abs{z_j}<1. 
\label{z8}
\]
In this case, the model space $\Ran H_{\theta}$ can be easily described: 
$$
\Ran H_{\theta}=\{P(z)/D(z),\quad P\in\bbC_{k}[z]\},
$$
where $\bbC_{k}[z]$ is the space of all polynomials of degree $\leq k$ and 
$D(z)$ is the denominator in \eqref{z8}, i.e. 
$$
D(z)=\prod_{j=1}^k(1-\overline{z_j}z).
$$
The action of $H_\theta$ on $\Ran H_{\theta}$ in this case is given by 
\[
H_\theta: 
P(z)/D(z)\mapsto e^{i\alpha}z^k\overline{P(z)}/D(z), \quad z\in \bbT. 
\label{z9}
\]

\subsection{Isometric multipliers and weighted model spaces}\label{sec.z5}

The following notion will be used throughout the paper. 
It appeared in the literature in connection with the characterization of kernels of Toeplitz operators, see
\cite{Hitt,Hayashi,Sarason}. 
Let $F\subset\calH^2(\bbT)$ be a closed subspace; 
a holomorphic function $p$ in the unit disk  is called an {\it  isometric  multiplier} on $F$, 
if for every $f\in F$, the product  $pf$ belongs to $\calH^2(\bbT)$ 
and
\[
\norm{pf}=\norm{f}.
\label{z10}
\]
In this case, we write 
$$
pF:=
\{pf: f\in F\}.
$$
We will mostly use this notion for $F=\calK_\theta$, where $\theta$ is an inner function.
In this case, we will call $p\calK_\theta$ a \emph{weighted model space}. 
(This is not a standard piece of terminology; the adjective \emph{weighted} often implies
the choice of a norm different from the standard one, but here in fact the norm is unchanged.)

Observe that if $\theta(0)=0$, then  $\1\in\calK_\theta$, and therefore in this case we necessarily have
$p\in\calH^2(\bbT)$ and $\norm{p}=1$. 
In \cite{Sarason}, 
Sarason characterized all isometric multipliers of a given model space $\calK_\theta$ with $\theta(0)=0$, 
proving that these are the functions $p\in \calH^2(\bbT)$ of norm one, of the form
\[
p(z)=\frac{a(z)}{1- \theta (z)b(z)}\ ,
\label{z10ab}
\]
where $a, b\in \calH^\infty$ are such that $|a|^2+|b|^2=1$ almost everywhere on the unit circle. 

\begin{remark}\label{rmk.z3}
Let $M=p\Ran H_\theta$, where $p$ is an isometric multiplier on $\Ran H_\theta=\calK_{z\theta}$. 
Then the parameters $p$ and $\theta$ in this representation for $M$ 
are uniquely defined up to unimodular multiplicative constants. 
This is an easy fact, which follows from Theorem~\ref{thm.z2a}(ii) below. 
\end{remark}

\subsection{Main result: the structure of  Schmidt subspaces}\label{sec.z6}

For a symbol $a\in L^\infty(\bbT)$, we denote by $T_a$ the Toeplitz operator in $\calH^2(\bbT)$,
defined by $T_a f=P(a f)$. In fact, below we use this notation also for unbounded
symbols, but due to the isometricity of the corresponding multipliers, the resulting
Toeplitz operators turn out to be bounded on relevant domains.

The following theorem is the main result of the paper. 
We recall that the total multiplicity of the spectrum of a self-adjoint operator $A$ in an interval $\Delta$ 
is the rank of the spectral projection $\mathcal E_A(\Delta)$. 
\begin{theorem}\label{thm.z0}
Let $u\in\BMOA(\bbT)$, let $H_u$ be the anti-linear Hankel operator \eqref{z0}, and 
let $s>0$ be a singular value of $H_u$, i.e. $E_H(s)\not=\{0\}$. 
\begin{enumerate}[\rm (i)]
\item
There exists an inner function $\theta$ and an isometric multiplier $p\in\calH^2(\bbT)$ 
on $\Ran H_\theta$ 
such that 
\[
E_H(s)=p\Ran H_\theta.
\label{z10a}
\]
The functions $p$ and $\theta$  can be chosen such that
$$
H_u(p)=s\theta p,
$$
in which case  we have
$$
H_u T_p=sT_pH_\theta \quad \text{ on $\Ran H_\theta$.}
$$
\item
Let $\varphi$ be the inner factor of $p$.  Then 
we have $$s=\norm{H_u|_{\varphi\calH^2}}$$ and 
the degree of $\varphi$ equals the total multiplicity of the spectrum of the self-adjoint operator
$\abs{H_u}=\sqrt{H_u^2}$ in the open interval $(s,\infty)$. 
\end{enumerate}
\end{theorem}

\begin{remark}
\begin{enumerate}
\item
Equation $H_u(p)=s\theta p$ is simply a normalisation condition on the unimodular
constants in the definition of $p$, $\theta$ (see Remark~\ref{rmk.z3}). 
\item
Part (i) of the theorem shows that the parameters $p$, $\theta$ describe both the structure of the Schmidt subspace $E_H(s)$ 
and the action of $H_u$ on this subspace. 
In fact, the action of $H_u$ on $E_H(s)$ is reduced to the 
simple action of $H_{\theta}$. 
\item
Theorem~\ref{thm.z0}(i)
extends part (i) of the AAK theorem in two directions. 
Firstly, Theorem~\ref{thm.z0}(i) applies to both finite and infinite dimensional 
subspaces. Secondly, it provides the isometry of the corresponding  multipliers. 
\item
The inner function $I$ in \eqref{aak1} can be identified with the inner factor $\varphi$ of $p$ 
in the representation \eqref{z10a}. 
Further, we recognise the action \eqref{z9} in formula \eqref{aak2}. 
\item 
For compact $H_u$, part (i) of the theorem 
in a less explicit form
(and without the isometry property \eqref{z10}) 
appeared earlier in \cite{GGAst}. 
Of course, for compact $H_u$, all Schmidt subspaces are finite dimensional; the 
action of $H_u$ was expressed in the form \eqref{z9} in \cite{GGAst}. 
The proof in \cite{GGAst} heavily relied on compactness and on the 
use of the AAK theory. The proof we give in this paper is both simpler and 
more general.  
\item
If $s=\norm{H_u}$, part (ii) of the theorem says that the degree of $\varphi$ is zero, i.e. $p$ is
an outer function. 
If $\abs{H_u}$ has infinitely many eigenvalues or some essential spectrum in the interval $(s,\infty)$, 
it says that the degree of $\varphi$ is infinite. 

\item
Our proof of part (i) of the theorem is independent from the AAK theory. 
It involves the analysis of the Schmidt subspaces of both $\Gamma$ and $\wt \Gamma$ and 
also borrows some elements  from the theory of nearly $S^*$-invariant subspaces. 
On the other hand, the proof of part (ii) is essentially the adaptation of the original AAK argument
to the language of weighted model spaces, with some simplifications. 
\end{enumerate}
\end{remark}

\begin{remark}\label{rmk.z7}
A natural question is whether any subspace of the form $p\Ran H_{\theta}$ 
can appear as a Schmidt subspace of a bounded Hankel operator.
Let $p=\varphi p_0$ be the inner-outer factorisation of $p$; thus, a subspace 
$p\Ran H_\theta$ is characterised by the triple $\theta$, $\varphi$, $p_0$. 
In Section~\ref{sec.example} we show that any  pair of inner functions $\theta$, $\varphi$ can appear
in a description of a Schmidt subspace of a bounded Hankel operator. 
We don't know whether there are any 
restrictions on the outer factor $p_0$, apart from Sarason's condition \eqref{z10ab}.
\end{remark}

\subsection{The self-adjoint case}\label{sec.z7}
Here we consider the case of selfadjoint Hankel operators and describe the corresponding eigenspaces. 
Clearly,  a Hankel matrix $\Gamma$ is self-adjoint if and only if all coefficients $\gamma_j$ are real. 
Thus, we  consider the Hankel operators $H_u$ with the symbol $u=\wc \gamma$ having real Fourier coefficients. 
Then $H_u$ maps the space of functions $\calH^2_{\rm real}$ with \emph{real} Fourier
coefficients to itself, and the restriction of $H_u$ onto this space is unitarily equivalent to 
$\Gamma$ acting on the $\ell^2$ space of real sequences. 
The following corollary characterizes the eigenspaces of this restriction associated to non-zero eigenvalues.
\begin{corollary}[the selfadjoint case]\label{cor.z0}
Consider $H_u$ with $u\in\BMOA(\bbT)\cap\calH^2_{\rm real}$ acting on $\calH^2_{\rm real}$. 
 Let $\lambda \in \bbR\setminus \{ 0\} $ be an eigenvalue of $H_u$. Then there exists an inner function $\theta \in \calH^2_{\rm real}$ and an isometric multiplier $p \in \calH^2_{\rm real}$ on $\Ran H_\theta $ such that
 $$\Ker (H_u-\lambda I)=\{ pf : f\in \Ran H_\theta \cap \calH^2_{\rm real}\ ,\ H_\theta f= f\} \ .$$
\end{corollary}
Under the hypothesis of the corollary, we also have
$$
\Ker (H_u+\lambda I)=\{ pf: f\in \Ran H_\theta\cap \calH^2_\real\ , \ H_\theta f=-f\}. 
$$
Since 
$$
\Ker (H_u\pm\lambda I)=p\Ker(H_\theta\pm I),
$$
the analysis of the dimension of these eigenspaces reduces to the 
analysis of the action of $H_\theta$ on $\Ran H_\theta$. If $\deg \theta=\infty$, it is clear
that both dimensions above are infinite. If  $\deg \theta=k<\infty$, then 
\eqref{z9} shows that the action of $H_\theta$ on $\Ran H_\theta$ can be represented by 
the $(k+1)\times(k+1)$ matrix
$$
\begin{pmatrix}
0 & 0 & \cdots & 0 & 1
\\
0 & 0 & \cdots & 1 & 0
\\
\vdots & \vdots & & \vdots & \vdots 
\\
0 & 1 & \cdots & 0 & 0
\\
1 & 0 & \cdots & 0 & 0
\end{pmatrix}. 
$$
Thus, we are led to considering the eigenspaces 
of this matrix corresponding to the eigenvalues $\pm1$. 
As in \cite{GGEsterle}, we conclude that 
$$
\abs{\dim\Ker (H_u-\lambda I)-\dim\Ker (H_u+\lambda I)}\leq1,
$$
which recovers the result of \cite{MPT}. 
 
\subsection{Hankel operators on the real line}\label{sec.z8}
Here we briefly discuss the analogous result for Hankel operators acting on the Hardy space $\calH^2(\bbR)$ of the real line. 
Let us first introduce some notation. 
As usual, 
$$
\calH^2(\bbR)=\{\bbf\in L^2(\bbR): \wh \bbf(\xi)=0, \,\, \xi<0\}, 
$$
where $\wh \bbf$ is the Fourier transform,
$$
\wh\bbf(\xi)=\int_\bbR e^{-i2\pi \xi x}\bbf(x)dx.
$$
As it is standard, we shall consider functions in $\calH^2(\bbR)$ as holomorphic functions in the upper half-plane. 
Let $\bP$  be the orthogonal projection from $L^2(\bbR)$ onto $\calH^2(\bbR)$. 
For $\bu\in\BMOA(\bbR)$, the Hankel operator $\bH_\bu$ in $\calH^2(\bbR)$ is defined by 
$$
\bH_\bu \bbf=\bP(\bu \overline{\bbf}), \quad \bbf\in \calH^2(\bbR). 
$$
Since the kernel of $\bH_\bu$ is invariant under the shifts $\bbf(x)\mapsto e^{ix\xi}\bbf(x)$, $\xi>0$,
by a theorem of Lax \cite{Lax} (see also \cite[Section 6.5]{Nikolski}), similarly to the unit disk case, 
this kernel can be described as
$$
\Ker \bH_\bu=\bpsi\calH^2(\bbR), 
$$
where $\bpsi$ is an inner function in the upper half-plane.

For an inner function $\btheta$ in the upper half-plane, we will use the model space
$\bcalK_\theta=\Ran \bH_\btheta\subset \calH^2(\bbR)$, characterised by the condition
$$
\bbf\in \Ran \bH_\btheta \text{ if and only if } \btheta\overline{\bbf}\in\calH^2(\bbR).
$$
A holomorphic function $\bp$ in the upper half plane is called an isometric multiplier on $\Ran \bH_\btheta$, if
for any $\calH^2(\bbR)$, the product $\bp\bbf$ belongs to $\calH^2(\bbR)$ and 
$$
\norm{\bp\bbf}=\norm{\bbf}. 
$$
In this case, we shall call 
$$
\bp\Ran \bH_\btheta=\{\bp\bbf: \bbf\in \Ran \bH_\btheta\}
$$
the weighted model space, associated with $\bp$ and $\btheta$; clearly, this is a closed
subspace of $\calH^2(\bbR)$. 
We shall denote by $\bT_\bp$ the Toeplitz operator with the symbol $\bp$ acting on $\Ran\bH_\btheta$, 
i.e. $\bT_\bp\bbf=\bp\bbf$. 

In analogy with Theorem~\ref{thm.z0}, we have the following result. 

\begin{theorem}\label{thm.bz0}
Let $\bu\in\BMOA(\bbR)$, let $\bH_\bu$ be the corresponding anti-linear Hankel operator and 
let $s>0$ be a singular value of $\bH_\bu$, i.e. $\Ker (\bH_\bu^2-s^2I)\not=\{0\}$. 
\begin{enumerate}[\rm (i)]
\item
There exists an inner function $\btheta$ in the upper half-plane and an isometric multiplier $\bp$ 
on $\Ran \bH_\btheta$ 
such that 
$$
\Ker (\bH_\bu^2-s^2I)=\bp\Ran \bH_\btheta.
$$
The functions $\bp$ and $\btheta$  can be chosen such that
$$
\bH_\bu \bT_\bp=s\bT_\bp\bH_\btheta \quad \text{ on $\Ran \bH_\btheta$.}
$$
\item
Let $\bvarphi$ be the inner factor of $\bp$.  Then 
we have 
$$
s=\norm{\bH_\bu|_{\bvarphi\calH^2}}
$$ 
and 
the degree of $\bvarphi$ equals the total multiplicity of the spectrum of 
$\abs{\bH_\bu}=\sqrt{\bH_\bu^2}$ in the open interval $(s,\infty)$. 
\end{enumerate}
\end{theorem}

\subsection{The structure of the paper}
The strategy of the proof is outlined in Section~\ref{sec.bb}, 
where we introduce the Hardy space anti-linear version $K_u$ of the operator $\wt \Gamma=\{\gamma_{j+k+1}\}_{j,k=0}^\infty$ and
state Theorem~\ref{thm.z1}, which is a more detailed version of our main result, involving both $H_u$ and $K_u$.
Further, we discuss a generalisation of nearly $S^*$-invariant subspaces and state a general geometric
result (Theorem~\ref{thm.z2}) about such subspaces. 

In Section~\ref{sec.aa} we prove Theorem~\ref{thm.z2}. 
Using this theorem, in Section~\ref{sec.b} we prove Theorem~\ref{thm.z1}. 
In Section~\ref{sec.aaa} we prove  Corollary~\ref{cor.z0} and  Theorem~\ref{thm.z0}(i).
Section~\ref{sec.example} is devoted to a simple example, corresponding to Hankel operators with one or two singular values.  
In Section~\ref{sec.aab} we prove Theorem~\ref{thm.z0}(ii). 
In Section~\ref{app.B} by considering a conformal map from the unit disk to the upper half-plane we prove Theorem~\ref{thm.bz0}. 
In Appendix~\ref{sec.app}, we reformulate our main result Theorem~\ref{thm.z0}(i) in terms of the linear 
(rather than anti-linear) representation of Hankel operators in the Hardy space.

\subsection{Acknowledgements}
Our research was supported by EPSRC grant ref. EP/N022408/1. 
We acknowledge the hospitality of the Mathematics Departments at King's College and at Universit\'e Paris-Sud XI. 
We are grateful to Alexandru Aleman, Jonathan Partington, Evgeny Abakumov and Roman Bessonov 
for useful discussions concerning the topic of weighted model spaces.

\section{The strategy of the proof}\label{sec.bb}

\subsection{The operator $K_u$}
Our proof  requires another Hankel operator, which is defined by 
$$
K_u:=H_uS=S^*H_u=H_{S^*u}. 
$$
Clearly, $K_u$ is unitarily equivalent to $\wt \Gamma\calC$, where $\wt \Gamma=\{\wh u_{j+k+1}\}_{j,k=0}^\infty$. 
We have a crucial identity
\[
K_u^2=H_uSS^*H_u=H_u^2-(\cdot,u)u\ ,
\label{a5}
\]
where $(\cdot, u)u$ denotes the rank one operator corresponding to the element
$u$. 
For $s>0$, similarly to $E_{H_u}(s)$, we denote
$$
E_{K_u}(s)=\Ker(K_u^2-s^2I). 
$$
We shall write $E_K(s)$ instead of $E_{K_u}(s)$ 
when the choice of $u$ is clear from the context. 
We start with the basic statement which shows that (as a consequence of \eqref{a5})
the eigenspaces $E_H(s)$ and $E_K(s)$ differ by the one-dimensional subspace 
spanned by $u$. 
\begin{lemma}\label{lma.z1}
Let $s>0$ be a singular value of either $H_u$ or $K_u$, i.e. 
$$
E_H(s)+E_K(s)\not=\{0\}.
$$
Then one (and only one) of the following properties holds: 
\begin{enumerate}[\rm (1)]
\item
$u\not\perp E_H(s)$, and $E_K(s)=E_H(s)\cap u^\perp$;
\item
$u\not\perp E_K(s)$, and $E_H(s)=E_K(s)\cap u^\perp$.
\end{enumerate}
\end{lemma}
In case (1) above, we will say that 
the singular value $s$ is \emph{$H$-dominant};
in case (2) we will say that $s$ is 
\emph{$K$-dominant}.

This lemma was established in \cite{GGAst} in the case when $H_u$ and $K_u$ are compact, 
but in fact the proof does not use compactness. 
In any case, in Section~\ref{sec.b} for completeness we repeat the proof. 

We note that for a general pair of operators $H_u^2$, $K_u^2$, satisfying the 
rank-one identity \eqref{a5}, three scenarios are possible: (1), (2) of Lemma~\ref{lma.z1} and 

(3) $u\perp E_H(s)$, $u\perp E_K(s)$, and $E_H(s)=E_K(s)$. 

\noindent
Identity $K_u=S^*H_u$ ensures that (3) is not possible, see the proof of Lemma~\ref{lma.z1}.

\subsection{ The Schmidt subspaces of $H_u$ and $K_u$}\label{sec.a6}

If $g\in \calH^2(\bbT)$, $g\not =0$, and if $g/\norm{g}$ is an isometric multiplier of $\Ran H_\theta$, 
below we write for simplicity 
$g\Ran H_\theta $ instead of $(g/\norm{g})\Ran H_\theta$.

The following theorem describes the Schmidt subspaces of both $H_u$ and $K_u$. 
The two cases below correspond to the two cases in Lemma~\ref{lma.z1}. 

\begin{theorem}\label{thm.z1}
Let $u\in\BMOA(\bbT)$ and let the anti-linear Hankel operators $H_u$, $K_u$ be as defined above. 
Let $s>0$ be a singular value of either $H_u$ or $K_u$. 
\begin{enumerate}[\rm (i)]
\item
Let $s$ be $H$-dominant (i.e.  $u\not\perp E_H(s)$); 
denote by $\1_s$ (resp. by $u_s$) the orthogonal projection of 
$\1$ (resp. of $u$) onto $E_H(s)$. 
Then $u_s=s\psi_s \1_s$ for an inner function $\psi_s$. 
The function $\1_s/\norm{\1_s}$ is an isometric multiplier on $\Ran H_{\psi_s}$ and
the subspaces $E_H(s)$, $E_K(s)$ are given by 
\begin{align}
E_H(s)&=\1_s\Ran H_{\psi_s}, 
\label{z6a}
\\
E_K(s)&=\1_s(\Ran H_{\psi_s}\cap \psi_s^\perp)=\1_s\Ran K_{\psi_s}. 
\notag
\end{align}
The action of $H_u$ and $K_u$ on these subspaces is given by 
\begin{align}
H_u T_{\1_s}&=sT_{\1_s}H_{\psi_s}\quad \text{ on $\Ran H_{\psi_s}$},
\label{z7}
\\
K_u T_{\1_s}&=sT_{\1_s}K_{\psi_s} \quad \text{ on $\Ran K_{\psi_s}=\Ran H_{\psi_s}\cap \psi_s^\perp$}.
\label{z4}
\end{align}
The inner function $\psi_s$ is uniquely defined by \eqref{z6a} and \eqref{z7}.

\item
Let $s$ be $K$-dominant (i.e. $u\not\perp E_K(s)$); 
denote by $\wt u_s$ the orthogonal projection of $u$ onto $E_K(s)$. 
Then 
$K_u(\wt u_s)=s\wt \psi_s \wt u_s$ 
for an inner function $\wt \psi_s$.
The function $\wt u_s/\norm{\wt u_s}$ is an isometric multiplier on $\Ran H_{\wt\psi_s}$. 
The subspaces $E_H(s)$, $E_K(s)$ are given by 
\begin{align}
E_H(s)&=\wt u_s S( \Ran K_{\wt \psi_s})=\wt u_s(\Ran H_{\wt \psi_s}\cap \1^\perp),
\notag
\\
E_K(s)&=\wt u_s \Ran H_{\wt\psi_s}.
\label{z6b}
\end{align}
The action of $H_u$ and $K_u$ on these subspaces is given by 
\begin{align}
H_uT_{\wt u_s}&=sT_{\wt u_s}SH_{\wt \psi_s}
\quad
\text{on $S( \Ran K_{\wt \psi_s})=\Ran H_{\wt \psi_s}\cap \1^\perp$,}
\label{z5}
\\
K_uT_{\wt u_s}&=s T_{\wt u_s}H_{\wt \psi_s}
\quad 
\text{on $\Ran H_{\wt \psi_s}$. }
\label{z6}
\end{align}
The inner function $\wt\psi_s$ is uniquely defined by \eqref{z6b} and \eqref{z6}. 
\end{enumerate}
\end{theorem}
Let us briefly explain how Theorem \ref{thm.z0} follows from Theorem \ref{thm.z1} (we give the full proof in Section~\ref{sec.aaa}). 
In the $H$-dominant case, the first part of Theorem \ref{thm.z1} immediately provides
the decomposition \eqref{z10a} with $p=\frac{\1_s}{\Vert \1_s\Vert}$ and $\theta =\psi_s$. 
In the $K$-dominant case, the second part of Theorem \ref{thm.z1} provides
$E_H(s)=\wt u_s S( \Ran K_{\wt \psi_s})$. 
It is not difficult to derive the desired decomposition  \eqref{z10a} from here. 
Uniqueness is an easy consequence of Remark~\ref{rmk.z3}.

\subsection{A generalisation of nearly $S^*$-invariant subspaces}\label{sec.a7a}
The proof of Theorem~\ref{thm.z1} is a consequence of the following general statement. 

\begin{theorem}\label{thm.z2}
Let $M$ be a closed subspace of $\calH^2(\bbT)$, and let $p,q\in M$ be
unit vectors such that 
\[
M\cap p^\perp=S(M\cap q^\perp). 
\label{z11}
\]
\begin{enumerate}[\rm (i)]
\item
Then $q=\theta p$ with an inner function $\theta$. 
\item
Assume in addition that for any $f\in \Ran H_\theta\cap \calH^\infty(\bbT)$ we have the implication 
\[
pf\in M, \quad 
f\perp \1\quad \Rightarrow \quad pf\perp p. 
\label{z12}
\]
Then $p$ is an isometric multiplier on $\Ran H_\theta$ and 
\[
M=p\Ran H_\theta\ . 
\label{z12a}
\]
\end{enumerate}
\end{theorem}

Both the statement and the proof of Theorem~\ref{thm.z2} are closely related 
to the theory of the \emph{nearly $S^*$-invariant subspaces}, see \cite{Hitt,Sarason}. 
These are the subspaces $M$ satisfying $f\in M\cap \1^\perp\Rightarrow S^*f\in M$. 
By \cite[Proposition~3]{Hitt}, every nearly $S^*$-invariant subspace $M$ is either
a weighted model space or a subspace of the form $M=\psi\calH^2(\bbT)$
with inner $\psi$.

However, a weighted model space is not necessarily nearly $S^*$-invariant. 
For example, 
if $p(0)=0$, we have $M\perp \1$, and so nearly $S^*$-invariance of $M$ would mean 
$S^*M\subset M$, which is only possible if $M=\{0\}$. 
In any case, $M$ is $S^*$-invariant on a subspace of codimension $1$, 
and so it shares many properties with nearly $S^*$-invariant subspaces. 

To prove Theorem~\ref{thm.z1}(i), we take
$$
M=E_H(s), \quad p=\1_s/\norm{\1_s},\quad  q=u_s/\norm{u_s}.
$$ 
To prove Theorem~\ref{thm.z1}(ii), we take 
$$
M=E_K(s), \quad p=\wt u_s/\norm{\wt u_s},\quad  q=K_u \wt u_s/\norm{K_u\wt u_s}.
$$
The crucial geometric property \eqref{z11} in both cases is established in Lemma~\ref{crucial} below. 

\subsection{Some preliminary identities}
Throughout the rest of the paper, we use the following identity:
\[
H_u(fg)=\proj(\overline{f}H_u g), \quad f\in \calH^\infty(\bbT), \quad g\in \calH^2(\bbT).
\label{b0}
\]
The proof is a direct calculation:
$$
H_u(fg)=\proj (u\overline{f}\overline{g})=\proj (\overline{f} u\overline{g})=\proj (\overline{f}\proj(u\overline{g}))=
\proj(\overline{f}H_u g).
$$
Of course, this identity can also be applied with $K_u$ in place of $H_u$. 
Identity \eqref{b0} can be alternatively written in terms of the Toeplitz operators $T_f$, $T_{\overline{f}}$ as 
\[
H_uT_f=T_{\overline{f}}H_u.
\label{b0a}
\]
We also use the fact that both $H_u$ and $K_u$ satisfy the following symmetry condition: 
$$
(H_uf,g)=(H_ug,f), \quad f,g\in \calH^2(\bbT),
$$
which follows directly from the definition of $H_u$.

\section{Proof of Theorem~\ref{thm.z2}}\label{sec.aa}
\subsection{Factorisation $q=\theta p$}
1)
For every $\zeta\in \bbD$, denote by $f_\zeta$ the reproducing kernel of 
$M$, namely $f_\zeta\in M$ and 
$$\forall h\in M, \quad h(\zeta)=(h, f_\zeta)\ .$$
Denote by $g_\zeta$ the orthogonal projection of $f_\zeta$ onto $M\cap q^\perp $, 
and by $h_\zeta$ the orthogonal projection of $f_\zeta$ onto $M\cap p^\perp $. 
By assumption, we have $h_\zeta=SS^*h_\zeta$, with $S^*h_\zeta\in M\cap q^\perp$, 
and $Sg_\zeta\in M\cap p^\perp$, hence
$$
h_\zeta(\zeta)=\zeta S^*h_\zeta(\zeta)=\zeta(S^*h_\zeta,f_\zeta)
=\zeta(S^*h_\zeta,g_\zeta)=\zeta(h_\zeta,Sg_\zeta)
=\zeta(f_\zeta,Sg_\zeta)=\abs{\zeta}^2\overline {g_\zeta(\zeta)}\ .
$$
Since $f_\zeta$ is the reproducing kernel of $M$, we have
$$
\norm{g_\zeta}^2=(g_\zeta, f_\zeta)=g_\zeta(\zeta)\ ,\ 
\norm{h_\zeta}^2=(h_\zeta, f_\zeta)=h_\zeta(\zeta),
$$
and so we conclude that 
\[
\norm{h_\zeta}^2
=\abs{\zeta}^2\norm{g_\zeta}^2
\leq \norm{g_\zeta}^2.
\label{b3b}
\]

2) 
Let us first prove that $q=\theta p$ with some $\theta\in \calH^\infty$, $\norm{\theta}_{\calH^\infty}\leq1$. 
Since $q$ and $p$ are non zero holomorphic functions on the unit disk $\bbD$, it suffices to prove that
$$\forall \zeta \in \bbD\ ,\  \abs{q(\zeta)} \leq \abs{p(\zeta)}\ .$$
We shall obtain this from \eqref{b3b} by a duality argument. 
 
Expanding $f_\zeta$ with respect to two 
orthogonal decompositions 
$$
M=(M\cap q^\perp) \oplus \Span\{q\}\ ;\ M=(M\cap p^\perp) \oplus \Span\{p\}\ ,
$$
we obtain 
\[
f_\zeta=g_\zeta+\lambda q=h_\zeta+\mu p\ .
\label{b3a}
\]
The constants $\lambda $ and $\mu $ satisfy 
$$
\| f_\zeta\|^2
=
\|g_\zeta\|^2+| \lambda |^2
=
\|h_\zeta\|^2+|\mu |^2.
$$
Using \eqref{b3b}, we  conclude that $\abs{\mu}\geq \abs{\lambda}$. 
Taking an inner product of \eqref{b3a} with $q$ and with $p$, we get
\begin{gather*}
q(\zeta)=(q,f_\zeta)=(q,\lambda q)=\overline{\lambda},
\\
p(\zeta)=(p,f_\zeta)=(p,\mu p)=\overline{\mu}.
\end{gather*}
It follows that $\abs{q(\zeta)}\leq\abs{p(\zeta)}$, as required. 

3)
Finally, let us check that $\theta$ is inner. 
Using $q=\theta p$,  we see that
$$
0=\norm{p}^2-\norm{q}^2=\int_{-\pi}^\pi(1-\abs{\theta(e^{it})}^2)\abs{p(e^{it})}^2\frac{dt}{2\pi}.
$$
As $\abs{\theta(e^{it})}^2\leq1$, we see that the integrand above vanishes for a.e. $t$. 
Since $p(e^{it})\not=0$ for a.e. $t$, we 
conclude that $\abs{\theta(e^{it})}^2=1$ for a.e. $t$. 

\subsection{$p$ is an isometric multiplier on $\Ran H_\theta$}

1)
Denote 
$$
F=\Span\{(S^*)^n \theta: n\geq0\},
$$
where $\Span$ denotes the set of all finite linear combinations. 
It is easy to see that 
$$
\Clos F=\Ran H_\theta.
$$ 
Let us prove that 
$$
pF\subset M. 
$$
Let $f\in F$; assume that $pf\in M$ and 
consider the element $p(f-f(0))\in M$.  
Since $f-f(0)\perp \1$, by the geometric assumption \eqref{z12} we have 
$p(f-f(0))\perp p$. From \eqref{z11} it follows that $p(f-f(0))=zg$ with some $g\in M\cap q^\perp$
and so $pS^*f=g\in M$. Thus, we have an implication
$$
f\in F, \quad pf\in M \quad \Rightarrow \quad
pS^*f\in M. 
$$
Since  $q=p\theta\in M$, we can apply this implication to $f=\theta$, then to $f=S^*\theta$, etc, 
to obtain $p(S^*)^n\theta\in M$ for all $n\geq0$.

2)
For $f\in F$, write 
$$
pf=pf(0)+p(f-f(0)). 
$$
By the orthogonality assumption \eqref{z12}, the two terms in the right hand side are orthogonal to one another, and so
\[
\norm{pf}^2=\abs{f(0)}^2+\norm{p(f-f(0))}^2
=\abs{f(0)}^2+\norm{pS^*f}^2. 
\label{c7}
\]
On the other hand, we obviously have 
$$
\norm{f}^2=\abs{f(0)}^2+\norm{S^*f}^2; 
$$
subtracting, we obtain
$$
\norm{pf}^2-\norm{f}^2=\norm{pS^*f}^2-\norm{S^*f}^2. 
$$
Thus, we have an implication
$$
f\in F, \quad \norm{pf}=\norm{f} \quad \Rightarrow\quad
\norm{pS^*f}=\norm{S^*f}. 
$$
Applying this to $f=\theta$, then to $f=S^*\theta$, etc., we obtain
\[
\norm{p(S^*)^n \theta}=\norm{(S^*)^n\theta}, \quad \forall n\geq0.
\label{c7b}
\]

3) In order to extend the above relation to linear combinations of elements
$(S^*)^n \theta$, we need to prove that the set 
$$
F_0=\{f\in F: \norm{pf}=\norm{f}\}
$$
is linear. 
Here we use the argument of \cite{Hitt}. 
As a first step, let us check the inequality 
\[
\norm{pf}\geq \norm{f}\quad \forall f\in F.
\label{c7a}
\]
Rewrite \eqref{c7} as
$$
\norm{pf}^2=\abs{\wh f_0}^2+\norm{pS^*f}^2
$$
and apply this to $S^*f$ in place of $f$: 
$$
\norm{pS^*f}^2
=\abs{\wh f_1}^2+\norm{p(S^*)^2f}^2. 
$$
Iterating and summing, we obtain
$$
\norm{pf}^2=\sum_{n=0}^N \abs{\wh f_n}^2
+\norm{p(S^*)^{N+1}f}^2
\geq 
\sum_{n=0}^N \abs{\wh f_n}^2. 
$$
Sending $N\to\infty$, we obtain \eqref{c7a}. 

4) 
Consider the linear operator
$$
T_{1/p}: pF\to F.
$$
By \eqref{c7a}, this is a contraction. 
It is straightforward to see that $F_0$ can be
characterised as the image of $\Ker (T_{1/p}^*T_{1/p}-I)$ under $T_{1/p}$; thus, $F_0$ is linear. 
Thus, the isometry relation \eqref{c7b} extends to all linear combinations of elements $(S^*)^n\theta$. 
In other words, we obtain that the map
$$
T_p: F\to pF\subset M
$$
is an isometry. Since $F$ is dense in $\Ran H_\theta$, this map extends as an isometry 
$$
T_p: \Ran H_\theta\to p\Ran H_\theta\subset M. 
$$

\subsection{Proof of $p\Ran H_\theta=M$}

Consider the subspace
$$
V=M\cap (p\Ran H_\theta)^\perp. 
$$
Our aim is to prove that $V=\{0\}$. 

1)
Let us first prove that $V$ 
is $S^*$-invariant. 
Let $h\in V$. Then $h\in M\cap p^\perp$, and so $S^*h\in M\cap q^\perp$. 
It suffices to check that 
$(S^*h,h')=0$ for all $h'=pf$, $f\in \Ran H_\theta\cap \calH^\infty$. 
For such $f$, write
$\overline{f}\theta=c\1+zw$ with $w\in \Ran H_\theta\cap \calH^\infty$;  then
$$
f=\overline{c}\theta+\overline{z}\overline{w}\theta,
$$
and so 
$$
h'=\overline{c}p\theta+p\overline{z}\overline{w}\theta
=\overline{c}q+p\overline{z}\overline{w}\theta. 
$$
Now 
$$
(S^*h,h')=c(S^*h,q)+(S^*h,p\overline{z}\overline{w}\theta). 
$$
Here the first term vanishes as $S^*h\in M\cap q^\perp$.
For the second term, we have $SS^*h=h$ and so 
$$
(S^*h,p\overline{z}\overline{w}\theta)
=
(h,p\overline{w}\theta)=0,
$$
because $h\perp p\Ran H_\theta$. 

2) Thus, $V$ is $S^*$-invariant and 
 $V\perp \1$ 
(as $V\subset M\cap p^\perp\subset S(M\cap q^\perp)$). 
But any subspace satisfying these two conditions is trivial, since 
for any $h\in V$ we have 
$$
\wh h(n)=(h,S^n \1)=((S^*)^nh,\1)=0
$$
for all $n$. Thus, $V=\{0\}$. 
The proof of Theorem~\ref{thm.z2} is complete.

\section{Proof of Theorem~\ref{thm.z1}}\label{sec.b}

\subsection{Lemmas on subspaces}

\begin{proof}[Proof of Lemma~\ref{lma.z1}]
1) 
Assume $u\not\perp E_H(s)$; let us prove that $u\perp E_K(s)$ and 
$$
E_K(s)=E_H(s)\cap u^\perp.
$$
Because $K_u^2=H_u^2-(\cdot, u)u$, it is clear that 
$$
E_H(s)\cap u^\perp= E_K(s)\cap u^\perp.
$$ 
In particular, $E_H(s)\cap u^\perp \subset E_K(s)$; 
let us prove the converse inclusion.
Let $h\in E_K(s)$, and let $h'\in E_H(s)$ be such that $(h', u)\neq 0$. Then
\begin{align*}
s^2(h, h')&=(K_u^2h, h')=(H_u^2h, h')-(h, u)(u,h')
\\
&=(h, H_u^2h')-(h, u)(u, h')=s^2(h,h')-(h, u)(u, h')\ ,
\end{align*}
which implies that $h$ is orthogonal to $u$. Hence $E_K(s)\perp u$, and consequently 
$$
E_K(s)=E_K(s)\cap u^\perp=E_H(s)\cap u^\perp\subset E_H(s),
$$ 
as required. 

2) 
Assume $u\not\perp E_K(s)$; let us prove that $u\perp E_H(s)$ and 
$$
E_H(s)=E_K(s)\cap u^\perp.
$$
By the same reasoning as above we have $E_K(s)\cap u^\perp\subset E_H(s)$. 
To prove the converse inclusion, let $h\in E_H(s)$ and let $h'\in E_K(s)$ 
with $(h',u)\not=0$. 
Then, as at the previous step, we have
$$
s^2(h,h')=(H_u^2h,h')=(K_u^2u,h')+(h,u)(u,h')
=
s^2(h,h')+(h,u)(u,h'),
$$
and so $h\perp u$. This gives the inclusion $E_H(s)\subset E_K(s)\cap u^\perp$. 

3)
Finally, let us prove that $u$ cannot be orthogonal to both $E_H(s)$ and $E_K(s)$. 
If it is, then 
$$
E_H(s)=E_H(s)\cap u^\perp=E_K(s)\cap u^\perp=E_K(s).
$$
Denote this subspace by $V$; let us prove that $V=\{0\}$. 
Both $H_u$ and $K_u$ are anti-linear isomorphisms on $V$. 
Given $h\in V$, write $h=H_uh'$ with $h'\in V$. Then $S^*h=K_uh'\in V$. Furthermore, 
$$
0=(h', u)=(h',H_u \1)=
(\1, H_u h')=(\1, h)\ .
$$ 
Hence $S^*(V)\subset V$ and $V\perp \1$. 
Thus, as at the last stage of the proof of Theorem~\ref{thm.z2} above, we obtain $V=\{0\}$. 
\end{proof}

The following lemma is fundamental for our construction. 
It will allow us to check the crucial ``geometric'' hypothesis $M\cap p^\perp=S(M\cap q^\perp)$ 
of Theorem~\ref{thm.z2}.
\begin{lemma}\label{crucial}
Let $E_H(s)+E_K(s)\not=\{0\}$. Then 
\begin{align*}
S(E_H(s)\cap u^\perp )&=E_H(s)\cap \1^\perp,\quad\text{ if $s$ is $H$-dominant,}
\\
S(E_K(s)\cap (K_u u)^\perp)&=E_K(s)\cap u^\perp,\quad\text{ if $s$ is $K$-dominant.}
\end{align*}
\end{lemma}
\begin{proof}
1) Let $s$ be $H$-dominant; 
let us check that $S E_K(s)\subset E_H(s)\cap\1^\perp$. 
Let $g\in E_K(s)$; we need to prove that $H_u^2 Sg=s^2 Sg$. In order to do this, observe that 
$$
S^*(H_u^2Sg-s^2Sg)=S^*H_u^2Sg-s^2 S^*Sg=K_u^2g-s^2g=0,
$$
and so $H_u^2Sg-s^2Sg\in\Ker S^*$, i.e. $H_u^2Sg-s^2Sg=c\1$. 
Let us compute $c$:
$$
c=(H_u^2Sg-s^2Sg,\1)=(H_u^2Sg,\1)=(H_u\1,H_uSg)=(u,K_ug)=0,
$$
because $K_ug\in E_K(s)$. We have proved that $Sg\in E_H(s)$. 
Finally, it is obvious that $S E_K(s)\perp\1$. 

2) Let $s$ be $H$-dominant; 
let us check that $E_H(s)\cap \1^\perp\subset SE_K(s)$. 
Let $h\in E_H(s)\cap \1^\perp$; then $h=Sg$ and $g=S^*h$. 
We need to check that $g\in E_K(s)$, i.e. that $K_u^2g=s^2 g$. 
We have
$$
K_u^2g=S^*H_u^2Sg=S^*H_u^2h=s^2S^*h=s^2g,
$$
as required. 

3) Let $s$ be $K$-dominant; 
let us check that $S(E_K(s)\cap (K_u u)^\perp)\subset E_H(s)$. 
Let $h\in E_K(s)\cap (K_u u)^\perp$ and $g=Sh$. In order to prove that $H_u^2g=s^2g$, consider
$$
S^*(H_u^2g-s^2g)=S^*H_u^2Sh-s^2 S^*g=K_u^2h-s^2h=0.
$$
It follows that $H_u^2g-s^2g=c\1$. Let us compute $c$: 
\begin{multline*}
c=(H_u^2g-s^2g,\1)=(H_u^2g,\1)=(H_u\1,H_ug)
\\
=(u,H_uSh)=(u,K_uh)=(h,K_uu)=0,
\end{multline*}
as required. 

4) Let $s$ be $K$-dominant;
let us check that $E_H(s)\subset S(E_K(s)\cap (K_u u)^\perp)$. 
Take $g\in E_H(s)$; let us first check that $(g,\1)=0$. Since $H_ug\in E_H(s)\subset u^\perp$, we have
$$
0=(u,H_ug)=(H_u\1,H_ug)=(H_u^2g,\1)=s^2(g,\1),
$$
as claimed. Thus, $g=Sh$; let us prove that $h\in E_K(s)$ and $h\perp K_u u$. 
To see the first inclusion, we compute
$$
K_u^2h-s^2h=S^*H_u^2Sh-s^2h=S^*H_u^2g-s^2h=s^2S^*g-s^2h=0.
$$
Finally, 
$$
(h,K_uu)=(u,K_uh)=(u,H_uSh)=(u,H_ug)=0.
$$
\end{proof}

\subsection{Proof of Theorem~\ref{thm.z1}(i)}

1)
Let $P_s$ be the orthogonal projection onto $E_{H}(s)$. Then $P_s$ commutes with $H_u$ and therefore
$$
u_s=P_su=P_sH_u\1=H_uP_s\1=H_u\1_s. 
$$
Furthermore, we have
$$
\norm{u_s}^2=(H_u\1_s,H_u\1_s)=(H_u^2\1_s,\1_s)=s^2\norm{\1_s}^2. 
$$
Next, let us apply Theorem~\ref{thm.z2}(i)  with $M=E_H(s)$, $p=\1_s/\norm{\1_s}$ and $q=u_s/\norm{u_s}$. 
The geometric hypothesis $M\cap p^\perp =S(M\cap q^\perp)$ is satisfied by 
Lemma~\ref{crucial}. 

Thus, we obtain 
$$
\frac{u_s}{\norm{u_s}}=\psi_s\frac{\1_s}{\norm{\1_s}}
$$
with some inner function $\psi_s$. This can be rewritten in the form
$$
u_s=s\psi_s \1_s. 
$$

2)
Let us apply Theorem~\ref{thm.z2}(ii) with the same parameters. 
The orthogonality assumption \eqref{z12} is evidently satisfied:
$$
(pf, p)=(pf,\1_s)/\norm{\1_s}=(pf,\1)/\norm{\1_s}=0
$$
if $f(0)=0$. 
This yields the description $E_H(s)=\1_s\Ran H_{\psi_s}$ of our Schmidt subspace
and the formula $\norm{f\1_s}=\norm{f}\norm{\1_s}$ for all $f\in \Ran H_{\psi_s}$. 

Let us  check formula \eqref{z7} for the action of $H_u$ on $E_H(s)$ 
It suffices to check it on the dense set of elements $f\in \Ran H_{\psi_s}\cap \calH^\infty$. 
We have, using \eqref{b0}
$$
H_uT_{\1_s}f=H_u(f\1_s)=P(\overline{f}H_u \1_s)=P(\overline{f}u_s)=sP(\overline{f}\psi_s\1_s). 
$$
Since $\overline{f}\psi_s\in \calH^\infty(\bbT)$, we obtain \eqref{z7}:
$$
H_uT_{\1_s}f=s\overline{f}\psi_s\1_s=s\1_sH_{\psi_s} f=sT_{\1_s}H_{\psi_s}f. 
$$

3) It remains to describe the subspace $E_K(s)$ and the action of $K_u$ on this subspace. 
By Lemma~\ref{lma.z1}, we have $E_K(s)=E_H(s)\cap u^\perp$; on the other hand, by 
Theorem~\ref{thm.z2}, 
multiplication by $\1_s/\norm{\1_s}$ is a unitary operator from $\Ran H_{\psi_s}$ to $E_H(s)$. 
Thus, for $f\in \Ran H_{\psi_s}$, the product $\1_sf$ belongs to $E_K(s)$ if and only if
$$
0=(\1_s f,u)=(\1_s f,u_s)=s(\1_s f,\1_s\psi_s)=s\norm{\1_s}^2(f,\psi_s), 
$$
i.e. if and only if $f\in \Ran H_{\psi_s}\cap \psi_s^\perp$. 
But $f\perp \psi_s$ can be equivalently rewritten as $\overline{f}\psi_s\perp\1$; 
this relation means that $f=H_{\psi_s}(zw)$ with some $w\in \Ran H_{\psi_s}$. 
Thus, 
$$
\Ran H_{\psi_s}\cap \psi_s^\perp=\Ran H_{\psi_s}S=
S^*\Ran H_{\psi_s}=\Ran K_{\psi_s}, 
$$
as claimed. 

Let us check formula \eqref{z4} for the action of $K_u$ on $E_K(s)$. 
Let $f\in \Ran H_{\psi_s}\cap \psi_s^\perp$; then $H_{\psi_s}f\perp \1$, and so 
$$
S^*(\1_s H_{\psi_s}f)=\1_s S^*H_{\psi_s} f=\1_s K_{\psi_s}f.
$$
Now using formula \eqref{z7} for the action of $H_u$ on $E_H(s)$ 
it follows that
$$
K_u(\1_sf)=S^*H_u(\1_s f)=sS^*(\1_s H_{\psi_s}f)=s\1_sK_{\psi_s}f, 
$$
as required.

\subsection{Proof of Theorem~\ref{thm.z1}(ii)}
1)
Let $P_s$ be the orthogonal projection onto $E_K(s)$. Then $P_s$ commutes with $K_u$
and therefore $P_sK_u u=K_u P_s u=K_u\wt u_s$. Furthermore, we have 
$$
\norm{K_u\wt u_s}^2=(K_u^2\wt u_s,\wt u_s)=s^2\norm{\wt u_s}^2. 
$$
Let us apply Theorem~\ref{thm.z2}(i) with $M=E_K(s)$, 
$p=\wt u_s/\norm{\wt u_s}$ and 
$q=K_u\wt u_s/\norm{K_u\wt u_s}$. 
The geometric condition $M\cap p^\perp=S(M\cap q^\perp)$ is satisfied 
by Lemma~\ref{crucial}. 
We obtain 
$$
\frac{K_u\wt u_s}{\norm{K_u\wt u_s}}=\wt \psi_s\frac{\wt u_s}{\norm{\wt u_s}}
$$
with some inner function $\wt \psi_s$. This can be rewritten as  
$$
K_u\wt u_s=s\wt \psi_s\wt u_s.
$$

2) 
We would like to apply Theorem~\ref{thm.z2}(ii) with the same parameters. 
We need to check the orthogonality assumption \eqref{z12}: 
for $f\in \Ran H_{\wt\psi_s}\cap\calH^\infty$, $\wt u_s f\in E_K(s)$ and
$f(0)=0$ implies $\wt u_s f\perp \wt u_s$. 
This is a little more complicated than the analogous step in the $H$-dominant case. 
Write $f=zw$, with $w=S^*f\in \Ran H_{\wt \psi_s}\cap \calH^\infty$. 
We have 
\[
H_u(\wt u_s f)=H_u(z\wt u_s w)=K_u(\wt u_s w)=
P(\overline{w}K_u \wt u_s)=
sP(\overline{w}\wt \psi_s \wt u_s). 
\label{b16}
\]
Since $\overline{w}\wt \psi_s\in \calH^\infty$, we have 
\[
H_u(\wt u_s f)=
s\wt u_s\wt \psi_s\overline{w}=
s\wt u_s\wt \psi_sz\overline{f}\perp\1.
\label{b17}
\]
Thus, 
$$
(\wt u_s f,\wt u_s)=(\wt u_s f, u)=(\wt u_s f, H_u\1)=(\1,H_u(\wt u_s f))=0,
$$
as required. 

3) 
Now we can  apply Theorem~\ref{thm.z2}(ii),
which gives $\norm{f\wt u_s}=\norm{f}\norm{\wt u_s}$ 
for all $f\in \Ran H_{\wt\psi_s}$ and 
$$
E_K(s)=\wt u_s\Ran H_{\wt\psi_s}.
$$
Let us check formula \eqref{z6} for the action of $K_u$ on $E_K(s)$; 
this is a calculation similar to the above: 
$$
K_u(\wt u_s f)
=P(\overline{f}K_u \wt u_s)
=sP(\overline{f}\wt \psi_s\wt u_s)
=s\overline{f}\wt \psi_s\wt u_s=sH_{\wt \psi_s}f\wt u_s,
$$
as required.

4) It remains to describe the subspace $E_H(s)$ and the action of $H_u$ on this subspace. 
By Lemma~\ref{lma.z1}, we have $E_H(s)=E_K(s)\cap \wt u_s^\perp$; on the other hand, 
multiplication by $\wt u_s/\norm{\wt u_s}$ 
is an isometry from $\Ran H_{\wt\psi_s}$ onto $E_K(s)$. 
Thus, for $f\in \Ran H_{\wt\psi_s}$, the product $f\wt u_s$ belongs to $E_H(s)$ if and only if
$$
(f\wt u_s,\wt u_s)=0\quad \Leftrightarrow \quad (f,\1)=0.
$$
Now it is easy to see that
$$
\Ran H_{\wt \psi_s}\cap\1^\perp
=
SS^*\Ran H_{\wt \psi_s}
=
S\Ran K_{\wt \psi_s}. 
$$

Finally, formula \eqref{z5} for the action of $H_u$ on $E_H(s)$
has already been checked in \eqref{b16}--\eqref{b17}. 
The proof of Theorem~\ref{thm.z1} is now complete.

\section{Proof of Theorem~\ref{thm.z0}(i) and of Corollary~\ref{cor.z0}}\label{sec.aaa}

\subsection{Frostman shifts and the $K$-dominant case}

Formulas for the transformation of model spaces under the Frostman shift seem to be folklore among 
experts; the precise result we need can be found in \cite{Crofoot}:

\begin{theorem}\cite[Theorem 10]{Crofoot}\label{thm.z2a}
\begin{enumerate}[\rm (i)]
\item
Let $\theta$ be an inner function on $\bbD$, and let $w\in \bbD$. 
Let 
$$
\alpha_w(z)=\frac{w-z}{1-\overline{w}z}, 
\quad
g_w(z)=\frac{\sqrt{1-\abs{w}^2}}{1-\overline{w}z}. 
$$
Then $g_w\circ\theta$ is an isometric multiplier from $\calK_\theta$ onto $\calK_{\alpha_w\circ\theta}$, 
i.e. 
$$
(g_w\circ\theta) \calK_\theta=\calK_{\alpha_w\circ\theta}. 
$$
\item
Let $\theta_1$, $\theta_2$ be non-constant inner functions on $\bbD$. 
If there exists an isometric multiplier from $\calK_{\theta_1}$ onto $\calK_{\theta_2}$, then 
$\theta_1=\alpha\circ\theta_2$ for some disk automorphism $\alpha$. 
If $p_1$ and $p_2$ are two such multipliers, then $p_2=\gamma p_1$ for 
a unimodular constant $\gamma$. 
\end{enumerate}
\end{theorem}

In particular, if both $\theta_1$ and $\theta_2$ in Theorem~\ref{thm.z2a}(ii) 
vanish at the origin, then $\theta_1=\gamma'\theta_2$ for a unimodular
constant $\gamma'$. 
It follows that the parameters $p$ and $\theta$ in the weighted model space $p\Ran H_\theta$
are uniquely defined up to unimodular multiplicative constants, as claimed in Remark~\ref{rmk.z3} above.

\begin{proof}[Proof of Theorem~\ref{thm.z0}(i)]
As already explained, if $s$ is $H$-dominant, then the first part of Theorem~\ref{thm.z1}
immediately provides the required decomposition \eqref{z10a} with 
$p=\1_s/\norm{\1_s}$ and $\theta=\psi_s$. 
Uniqueness of $\psi_s$ is an easy consequence of Remark~\ref{rmk.z3}. 
Indeed, by this Remark, if $\psi_s'$ is another inner function satisfying
\eqref{z6a}, then $\psi_s'=e^{i\alpha}\psi_s$ for some $\alpha\in\bbR$. 
By \eqref{z7}, we have
$$
T_{\1_s}H_{\psi_s}=T_{\1_s}H_{\psi_s'} \text{ on $\Ran H_{\psi_s}$,}
$$
which implies $\psi_s=\psi_s'$.

Let us consider the $K$-dominant case. 
We first observe that for an inner function $\psi$, 
$$
\Ran K_\psi=S^*\Ran H_{\psi}=S^*\calK_{z\psi}=\calK_\psi. 
$$
Next, taking $w=\psi(0)$ in Theorem~\ref{thm.z2a}(i), we obtain
$$
\calK_\psi=\frac1{g_w\circ\psi}\calK_{\alpha_w\circ\psi}=\frac{1-\overline{\psi(0)}\psi}{\sqrt{1-\abs{\psi(0)}^2}}\Ran H_{S^*(\alpha_w\circ\psi)}, 
$$
with the multiplier $1/(g_w\circ\psi)$ acting isometrically on the model space in the right hand side. 
Applying this to $\psi=\wt\psi_s$, by Theorem~\ref{thm.z1}(ii) we arrive at
$$
E_H(s)=\wt u_s S\Ran K_{\wt\psi_s}=p\Ran H_\theta, 
$$
$$
\theta(z)=\frac{\wt\psi_s(z)-\wt\psi_s(0)}{z\bigl(1-\overline{\wt\psi_s(0)}\wt\psi_s(z)\bigr)}, 
\quad
p(z)=z\wt u_s(z)\bigl(1-\overline{\wt\psi_s(0)}\wt\psi_s(z)\bigl), 
$$
where $p/\norm{p}$ is an isometric multiplier on $\Ran H_\theta$. 

Let us check formula for the action of $H_u$ on $E_H(s)$. By \eqref{z5}, we have
$$
H_u(\wt u_s f)=s\wt u_s SH_{\wt \psi_s} f, \quad f\in S\Ran K_{\wt\psi_s}. 
$$
Write
$$
f=z\bigl(1-\overline{\wt\psi_s(0)}\wt \psi_s\bigr)h, 
\quad h\in \Ran H_\theta.
$$
Now we have
$$
SH_{\wt\psi_s}f
=z\wt \psi_s\overline{f}
=\wt\psi_s(1-\wt\psi_s(0)\overline{\wt\psi_s})\overline{h}
=(\wt\psi_s-\wt\psi_s(0))\overline{h}
=z(1-\overline{\wt\psi_s(0)}\wt \psi_s)\theta\overline{h}.
$$
Thus, we obtain
$$
H_u(ph)
=H_u(\wt u_s f)
=s\wt u_s  z(1-\overline{\wt\psi_s(0)}\psi_s)\theta\overline{h}
=sp\theta\overline{h}, 
$$
as required. Uniqueness follows from Remark~\ref{rmk.z3} as in the first
part of the proof.
\end{proof}

\subsection{The self-adjoint case}

\begin{proof}[Proof of Corollary~\ref{cor.z0}]
Let $s=\abs{\lambda}$; then 
$$
\Ker (H_u-\lambda I)\subset E_{H}(s). 
$$
Clearly, $H_u$ maps $\calH^2_\real$ to $\calH^2_\real$, and therefore the spectral projections
of $H_u^2$ satisfy the same property. In particular, the orthogonal projection $P_s$ onto $E_{H}(s)$
satisfies this property. 
By Theorem~\ref{thm.z0}, we have
$$
E_H(s)=p\Ran H_\theta
$$
for some $\theta$ and $p$. We recall that by Remark~\ref{rmk.z3}, the functions
$\theta$ and $p$ are uniquely defined up to unimodular multiplicative constants. 
Let us check that this constant in the definition of $p$ can be chosen such that $p\in\calH^2_\real$. 

Let $m$ be the smallest integer such that $p^{(m)}(0)\not=0$; then we have
$p(z)=z^m w(z)$ with $w\in \calH^2$ and $w(0)\not=0$.  
Then for any $f\in \Ran H_\theta$, 
setting $g=pf$, we get 
$$
(g,\overline{w(0)}p)=w(0)(pf,p)=w(0)(f,\1)=w(0)f(0)=(z^m wf,z^m)=(g,z^m), 
$$
which shows that $\overline{w(0)}p=P_sz^m$. 
Thus, we can choose
$$
p=\frac{P_s z^m}{\norm{P_sz^m}}
$$
which gives $p\in\calH^2_\real$. 

Now let us choose the unimodular constant in the definition of $\theta$ such that 
$$
H_u(p)=\lambda p\theta;
$$
then $\theta\in \calH^2_\real$ and on $\Ran H_\theta$ we have 
$$
H_uT_p=\lambda T_pH_\theta.
$$
The elements of $\Ker (H_u-\lambda I)\cap \calH^2_\real$ are 
exactly the elements of the form $pf$, where $f\in\Ran H_\theta\cap \calH^2_\real$ satisfies
$$
H_uT_pf=\lambda T_pf, 
$$
which is equivalent to the condition $H_\theta f=f$. 
\end{proof}

\section{An Example}\label{sec.example}

Here we take a break from our main line of exposition to discuss a simple special case
when $H_u$ and $K_u$ have only one or two singular values.  
In this case, it is easy to describe all relevant subspaces in an elementary way. 
In a companion paper \cite{paperB}, we consider the case when 
both $H_u$ and $K_u$ have finitely many singular values and give explicit formulas
for the symbol in terms of the singular values and the inner functions $\psi_s$, $\wt\psi_s$ appearing in 
Theorem~\ref{thm.z1}. 

\begin{theorem}\label{thm.ex}
Let $s$, $\wt s$ be positive numbers and $\psi$, $\wt \psi$ be inner functions. 
The following conditions are equivalent for $u\in \BMOA(\bbT)$: 
\begin{enumerate}[\rm (i)]
\item
The pair $(H_u,K_u)$ has precisely two positive singular values $s$, $\wt s$, 
with $s$ being $H$-dominant and $\wt s$ being $K$-dominant, and 
$\psi=\psi_s$ and $\wt\psi=\wt\psi_{\wt s}$ are the parameters of  Theorem~\ref{thm.z1}. 
\item
We have $s>\wt s>0$ and 
$$
u(z)=\frac{(s^2-\wt s\,^2)\psi(z)}{s-\wt s z\psi(z)\wt \psi(z)}, \quad z\in \bbD. 
$$
\end{enumerate}
\end{theorem}
\begin{proof}
Assume (i). 
Since $s$ is the only $H$-dominant singular value, the spectral decomposition of $H_u^2$ implies
$u=u_s\in E_{H}(s)$, hence $u=s\psi\1_s$. 
Multiplying $H_u^2\1_s=s^2\1_s$ by $\psi$, we obtain 
\[
\psi H_u u=su. 
\label{ex.1}
\]
Applying the identity $K_u^2h=H_u^2h-(h,u)u$ to $h=u$, we infer
$$
K_u^2 u=(s^2-\norm{u}^2)u. 
$$
It follows that $u=\wt u_{\wt s}$ and 
\[
s^2-\norm{u}^2=\wt s\,^2.
\label{ex.2}
\]
This gives $s>\wt s$. 
Further, by our assumption, it follows that 
\[
K_u u=\wt s\wt \psi u. 
\label{ex.3}
\]
Applying $SS^*f=f-(f,\1)\1$ to $f=H_u u$, we obtain
$$
SK_u u=H_u u-\norm{u}^2\1. 
$$
Multiplying this identity by $\psi$ and applying \eqref{ex.1}, \eqref{ex.2}, \eqref{ex.3}, 
we conclude
$$
\wt sz\psi(z)\wt\psi(z)u(z)=su(z)-(s^2-\wt s\,^2)\psi(z). 
$$
This gives the formula for $u$ in (ii). 

Assume (ii). 
Denote 
$$
h(z)=\frac1{s-\wt sz\psi(z)\wt\psi(z)}. 
$$
Let us compute $H_u h$. Performing the computations on the unit circle $\abs{z}=1$, we have
\begin{align*}
H_u h&=(s^2-\wt s\,^2)P\biggl(\frac{\psi}{(s-\wt s z\psi\wt\psi)(s-\wt s\overline{z}\overline{\psi}\overline{\wt \psi})}\biggr)
\\
&=(s^2-\wt s\,^2)P\biggl(\psi\, \frac{z\psi\wt\psi}{(s-\wt s z\psi\wt\psi)(sz\psi\wt\psi-\wt s)}\biggr).
\end{align*}
Applying the elementary identity
$$
\frac{(s^2-\wt s\,^2)\zeta}{(s-\wt s\zeta)(s\zeta-\wt s)}=\frac{s}{s-\wt s\zeta}+\frac{\wt s}{s\zeta-\wt s}
$$
to $\zeta=z\psi(z)\wt\psi(z)$, we observe that
\begin{align*}
(s^2-\wt s\,^2)\psi\, \frac{z\psi\wt\psi}{(s-\wt s z\psi\wt\psi)(sz\psi\wt\psi-\wt s)}
&=
\frac{s\psi}{s-\wt sz\psi\wt\psi}+\frac{\wt s\psi}{sz\psi\wt\psi-\wt s}
\\
&=
\frac{s\psi}{s-\wt sz\psi\wt\psi}+\frac{\wt s\overline{z}\overline{\wt\psi}}{s-\wt s\overline{z}\overline{\psi}\overline{\wt\psi}}.
\end{align*}
We note that the first term in the right hand side is in $\calH^2(\bbT)$ 
and the second one is orthogonal to $\calH^2(\bbT)$. 
Consequently, 
$$
H_u h=\frac{s\psi}{s-\wt sz\psi\wt\psi}=s\psi h=\frac{s}{s^2-\wt s\,^2}u.
$$
Using \eqref{b0}, we obtain 
$$
H_u^2 h=H_u(s\psi h)=sP(\overline{\psi}H_u h)=sP(\overline{\psi}s\psi h)=s^2 h, 
$$
and so $h,u\in E_H(s)$. Therefore $s$ is the only $H$-dominant singular value with the inner parameter of Theorem~\ref{thm.z1} being
$\psi_s=\psi$.
 
It remains to compute 
$$
K_u u=S^*H_u u=(s^2-\wt s\,^2)sS^*h. 
$$
Observe that 
$$
sh(z)=\frac{1}{1-\tfrac{\wt s}{s}z\psi\wt\psi}=1+\frac{\wt sz\psi\wt\psi}{s-\wt sz\psi\wt\psi}, 
$$
and so 
$$
sS^* h=\frac{\wt s\psi\wt\psi}{s-\wt sz\psi\wt\psi}.
$$
This gives 
$$
K_u u =\wt s\wt \psi u.
$$
Using \eqref{b0}, from here we get
$$
K_u^2 u=\wt s K_u(\wt \psi u)=\wt s P(\overline{\wt \psi} K_u u)=\wt s^2 u.
$$
It follows that $\wt s$ is the only $K$-dominant singular value with the inner parameter 
of Theorem~\ref{thm.z1} being $\wt\psi_{\wt s}=\wt \psi$. 
\end{proof}

Finally, we justify the claim made in Remark~\ref{rmk.z7}. 
Let $u$ be as in the above Theorem. 
Observe that $\wt u_{\wt s}=u$ and the inner-outer factorisation of $u$ is given by 
$$
u(z)=\psi(z)u_0(z), \quad u_0(z)=\frac{s^2-\wt s^2}{s-\wt s z\psi(z)\wt \psi(z)}. 
$$
Thus, by Theorem~\ref{thm.z1}(ii), we have the weighted model space representation
for the Schmidt space 
$$
E_{K_u}(s)=E_{H_{S^*u}}(s)=u\Ran H_{\wt\psi_s}=\psi u_0\Ran H_{\wt\psi}. 
$$
According to Theorem~\ref{thm.ex}, the inner functions $\psi$, $\wt\psi$ here can be chosen in an arbitrary way.

\section{Proof of Theorem~\ref{thm.z0}(ii)}\label{sec.aab}
Our arguments in this section follow closely the original AAK paper \cite{AAK3}; 
however, the last part of the proof is somewhat simpler, avoiding the perturbation 
argument of \cite{AAK3}. 
Throughout the section, we use the commutation relation \eqref{b0a}. 
We denote by $n(s;H_u)$ (resp. by $n[s;H_u]$) the total multiplicity of the spectrum of
the self-adjoint operator $\abs{H_u}$ in the interval $(s,\infty)$ (resp. $[s,\infty)$). 

\subsection{The case $s=\norm{H_u}$}

The first statement concerns the case $s=\norm{H_u}$; 
it contains Theorem~\ref{thm.z0}(ii) in the case $n(s;H_u)=0$. 

\begin{lemma}\label{lma.aab4}
Let $u\in \BMOA(\bbT)$, $s=\norm{H_u}>0$ and $E_H(s)\not=\{0\}$. 
If $f\in E_H(s)$ and $a$ is an inner divisor of $f$, then $f/a\in E_H(s)$ and 
\[
H_u(f/a)=T_aH_u f. 
\label{aab3}
\]
Furthermore, in the weighted model space representation $E_H(s)=p\Ran H_\theta$, the 
isometric multiplier $p$ is an outer function.
\end{lemma}
\begin{proof}
First observe that since $s=\norm{H_u}$, condition $f\in E_H(s)$ is equivalent to 
$$
\norm{H_u(f)}=s\norm{f}.
$$ 
Next, let $f=af_0\in E_H(s)$, where $a$ is inner and  $f_0\in \calH^2$.
We have 
$$
T_{\overline{a}}H_u f_0=H_uT_a f_0=H_u f
$$
and therefore 
$$
s\norm{f_0}=s\norm{f}
=
\norm{H_u f}
=
\norm{T_{\overline{a}}H_u f_0}
\leq
\norm{H_u f_0}; 
$$
it follows that $s\norm{f_0}=\norm{H_u f_0}$ and therefore $f_0\in E_H(s)$. 
Furthermore,  using the fact that 
$T_a T_{\overline{a}}$ is an orthogonal 
projection (as $T_a$ is an isometry), 
from 
$\norm{T_aT_{\overline{a}}H_u f_0}=\norm{H_u f_0}$
we conclude
$$
T_a T_{\overline{a}}H_u f_0=H_u f_0. 
$$
The last identity can be rewritten as
$$
T_a H_u f=H_u f_0,
$$
which is the same as \eqref{aab3}. 
Finally, applying the above statement to $f=p$, we obtain that the outer factor $p_0$ of $p$ is in 
$E_H(s)$, and so it can be represented as
$p_0=ph$ with $h\in \Ran H_\theta$. 
This implies that $p_0=p$. 
\end{proof}

\subsection{Introducing the AAK unimodular symbol}\label{sec.aab2}

\begin{lemma}\label{lma.aab2}
For $u\in \BMOA(\bbT)$ and $s>0$, let $f\in E_H(s)$, $f\not=0$. 
Then the ratio $\phi=H_u f/(s\overline{f})$ is a unimodular function 
on $\bbT$ which is independent of the choice of $f\in E_H(s)$.  
\end{lemma}
\begin{proof} 
In order to prove that $\phi$ is independent of the choice of $f$, take $f_1,f_2\in E_H(s)$
and let $H_u f_1 =sg_1$, $H_u f_2 =sg_2$; it suffices to check that
\[
\overline{f_1}g_2=\overline{f_2}g_1. 
\label{aab2}
\]
For $n\geq 0$, we have
\begin{multline*}
s(g_1\overline{f_2},z^n)
=
(H_u f_1,z^nf_2)
=
((S^*)^nH_u f_1,f_2)
=
(H_uS^nf_1,f_2)
\\
=
(H_uf_2,S^nf_1)
=
s(g_2\overline{f_1},z^n).
\end{multline*}
Similarly, we get
$$
(\overline{g_1}f_2,z^n)=(\overline{g_2}f_1,z^n), \quad n\geq0,
$$
and so \eqref{aab2} follows. Finally, the fact that $\abs{\phi}=1$ comes from 
applying \eqref{aab2} to $f_2=g_1$. 
\end{proof}

Let $\phi$ be as in the previous lemma and set $v=sP(\phi)$.
\begin{lemma}\label{lma.aab5}
Under the above conditions, 
$$
n(s;H_u)\leq 
\rank H_{u-v}\leq 
\deg\varphi.
$$
\end{lemma}
\begin{proof}
For any $f\in E_H(s)$ we have, by the definition of $\phi$, 
$$
H_v f=sP(\phi\overline{f})=H_u f
$$
and so the Hankel operator $H_{u-v}$ vanishes on the subspace $E_H(s)$. 
In particular, it vanishes on $p$ and consequently it vanishes on the minimal 
$S$-invariant subspace containing $p$. Since the inner factor of $p$ is $\varphi$, we conclude that
\[
\varphi\calH^2\subset \Ker H_{u-v}
\label{aab1}
\]
or equivalently 
$$
\Ran H_{u-v}\subset \calK_\varphi;
$$
it follows that 
$$
\rank H_{u-v}\leq \dim \calK_\varphi=\deg \varphi. 
$$
Further, we have 
$$
\norm{H_u-H_{u-v}}=\norm{H_v}\leq s\norm{\phi}_{L^\infty}=s.
$$
Putting this together, we obtain 
$$
n(s;H_u)=n(s;H_v+H_{u-v})\leq n(s;H_{v})+\rank H_{u-v}
\leq
0+ \deg \varphi=\deg\varphi,
$$
as required. 
\end{proof}

\subsection{Completing the proof}
Denote $w=T_{\overline{\varphi}} u$ and consider the Hankel operator
$$
H_w=H_uT_\varphi=T_{\overline{\varphi}}H_u.
$$
Observe that by \eqref{aab1} we have $H_w=H_vT_\varphi$ and so 
 $\norm{H_w}\leq \norm{H_v}\leq s$; in fact, by the following lemma we have $\norm{H_w}=s$. 
\begin{lemma}\label{lma.aab6}
We have 
$$
ap_0\Ran H_\theta\subset E_{H_w}(s)
$$
for any $a|\varphi$, i.e. for any inner divisor $a$ of $\varphi$. 
In particular, 
$$
E_{H_u}(s)=p\Ran H_\theta\subset E_{H_w}(s). 
$$
\end{lemma}
\begin{proof}
For every $f\in \Ran H_\theta$ we have
\begin{align*}
T_{\overline{\varphi}}H_u(pf)&=sT_{\overline{\varphi}}(pH_\theta f)= sT_{\overline{\varphi}}(\varphi p_0H_\theta f)
=sp_0H_\theta(f),
\\
T_{\overline{\varphi}}H_u(p_0H_\theta f)&=H_u(\varphi p_0H_\theta f)=H_u(pH_\theta f)=spH_\theta^2 f=
spf,
\end{align*}
and so $p\Ran H_\theta\subset E_{H_w}(s)$. Thus, we can apply Lemma~\ref{lma.aab4}
(with $H_w$ in place of $H_u$), which yields
that $ap_0f\in E_{H_w}(s)$ for any $a|\varphi$, 
as required. 
\end{proof}

At the last two steps of the proof below, we depart from the original argument of AAK. 
The AAK proof involves dimension count if $\deg\theta<\infty$ and a (rather tricky) approximation argument
if $\deg\theta=\infty$. Instead, we proceed by considering a quotient space, which works 
for all values of $\deg\theta$.

The following lemma is purely operator theoretic and does not use any specifics of Hankel operators. 
\begin{lemma}
With the above notation, we have
$$
\dim E_{H_w}(s)\ominus E_{H_u}(s)\leq n(s;H_u). 
$$
\end{lemma}
\begin{proof}
On the space $E=E_{H_w}(s)$ we consider the self-adjoint operator 
$$
A=P_E H_u^2|_E,
$$ 
where
$P_E$ is the orthogonal projection onto $E$. 
First note that by a standard variational argument, 
$$
n(s^2; A)\leq n(s^2; H_u^2).
$$
This follows, for example, by writing the minimax principle in the form \cite[Theorem 10.2.3]{BS}
$$
n(s^2;A)=\sup\{\dim L: L\subset E, \, (Af,f)>s^2\norm{f}^2\quad \forall f\in L\setminus\{0\}\},
$$
and comparing with a similar expression for $n(s^2; H_u^2)$.

Next, we notice that
$$
H_w^2=H_u T_\varphi T_{\overline{\varphi}}H_u\leq H_u^2
$$
and therefore $A\geq s^2 I$. 
Since $E_{H_u}(s)\subset E$, it is straightforward to see that 
$$
E_{H_u}(s)=\Ker (A-s^2 I). 
$$
Putting this together, we obtain 
$$
\dim E\ominus E_{H_u}(s)=n(s^2; A)\leq n(s^2;H_u^2)=n(s;H_u),
$$
as required. 
\end{proof}

Our final step is 
\begin{lemma}
With the above notation, we have
$$
\deg \varphi\leq
\dim E_{H_w}(s)\ominus E_{H_u}(s).
$$
\end{lemma}
\begin{proof}
Since  (by Lemma~\ref{lma.aab6})
$$
p_0\Span\{a: a|\varphi\}\Ran H_\theta\subset E_{H_w}(s)
$$
and $E_{H_u}(s)=p_0\varphi\Ran H_\theta$, we get, replacing the orthogonal complement by the algebraic quotient space,
\begin{align*}
\dim E_{H_w}(s)\ominus E_{H_u}(s)&=\dim E_{H_w}(s)/ E_{H_u}(s)
\\
&\geq\dim(p_0\Span\{a: a|\varphi\}\Ran H_\theta / p_0\varphi\Ran H_\theta).
\end{align*}
Since multiplication by $p_0$ is an injective operation, we have 
\begin{align*}
\dim(p_0\Span\{a: a|\varphi\}&\Ran H_\theta / p_0\varphi\Ran H_\theta)
\\
&=\dim(\Span\{a: a|\varphi\}\Ran H_\theta / \varphi\Ran H_\theta)
\\
&=\dim(\Span\{a: a|\varphi\}\Ran H_\theta \ominus \varphi\Ran H_\theta)
\\
&\geq\dim\Span\{a-(a,\varphi)\varphi: a|\varphi\},
\end{align*}
because for $a|\varphi$ and $f\in \calH^2(\bbT)$
\begin{multline*}
(a-(a,\varphi)\varphi,\varphi f)
=(a,\varphi f)-(a,\varphi)(\1,f)
=(\1,\overline{a}\varphi f)-(a,\varphi)(\1,f)
\\
=(\1,\overline{a}\varphi)(\1,f)-(a,\varphi)(\1,f)=0.
\end{multline*}
Finally, it is elementary to observe that
$$
\dim\Span\{a-(a,\varphi)\varphi: a|\varphi\}=\deg \varphi. 
$$
\end{proof}
\begin{proof}[Proof of Theorem~\ref{thm.z0}(ii)]
Putting together the last two lemmas, we obtain $\deg\varphi\leq n(s;H_u)$. 
Combining this with Lemma~\ref{lma.aab5}, we arrive at the desired conclusion
$$
\deg\varphi=n(s;H_u)=\rank H_{u-v}.
$$
\end{proof}

\section{Proof of Theorem~\ref{thm.bz0}}\label{app.B}
We consider the conformal map 
$$
\mu: \bbC_+\to\bbD, \quad \mu(z)=\frac{z-i}{z+i}, 
$$
and the corresponding unitary operator $U_\mu:\calH^2(\bbT)\to\calH^2(\bbR)$, 
\[
U_\mu f=\bbf, \quad \bbf(z)=\frac{1}{\sqrt{\pi}(z+i)}f(\mu(z)), \quad \Im z>0.
\label{B1}
\]
It is evident that $U_\mu$ maps a Beurling subspace $\theta\calH^2(\bbT)$ onto 
$(\theta\circ\mu)\calH^2(\bbR)$, and therefore 
$$
U_\mu\calK_\theta=\bcalK_{\theta\circ\mu}. 
$$
From here one reads off the formula for the action of $U_\mu$ on weighted model spaces: if 
$p$ is an isometric multiplier on $\calK_\theta$, then 
\[
U_\mu(p\calK_\theta)=\bp\bcalK_{\theta\circ\mu}, \quad \bp=p\circ \mu,
\label{B3}
\]
and $\bp$ is an isometric multiplier on $\bcalK_{\theta\circ\mu}$. 
Next, we have a straightforward
\begin{proposition}\label{prp.B3}
Let $\bu\in\BMOA(\bbR)$ and 
let $U_\mu$ be as in \eqref{B1}. Then we have
$$
U_\mu^*\bH_\bu U_\mu=H_{S^*u},
\quad
\bu(x)=u(\mu(x)).
$$
\end{proposition}
\begin{proof}
The following calculation is valid for the dense set of functions $h_1,h_2\in\calH^\infty(\bbT)$
and then the result extends to all $h_1,h_2\in \calH^2(\bbT)$:
\begin{multline*}
(\bH_\bu U_\mu h_1,U_\mu h_2)_{L^2(\bbR)}
=
(\bu,(U_\mu h_1)(U_\mu h_2))_{L^2(\bbR)}
\\
=
\int_{-\infty}^\infty \bu(x)\overline{h_1(\mu(x))}\overline{h_2(\mu(x))}\frac{dx}{\pi(x-i)^2}
=
\int_{-\infty}^\infty \frac{x+i}{x-i}\bu(x)\overline{h_1(\mu(x))}\overline{h_2(\mu(x))}\frac{dx}{\pi(x^2+1)}
\\
=
\int_{-\pi}^\pi e^{-i\theta} u(e^{i\theta})\overline{h_1(e^{i\theta})}\overline{h_2(e^{i\theta})}\frac{d\theta}{2\pi}
=
(H_{S^*u}h_1,h_2)_{L^2(\bbT)},
\end{multline*}
as required.
\end{proof}

\begin{proof}[Proof of Theorem~\ref{thm.bz0}]
Let us apply Proposition~\ref{prp.B3}; we obtain $U_\mu^*\bH_\bu U_\mu=H_{S^*u}=K_u$, with 
$u\circ\mu=\bu$. By Theorem~\ref{thm.z0}(i), we have
$$
E_{K_u}(s)=p\Ran H_\theta
$$
and 
\[
K_uT_p=sT_pH_\theta\text{ on }\Ran H_\theta,
\label{B2}
\] 
for some inner $\theta$ and for an isometric multiplier $p$ on $\Ran H_\theta$. 
Now let us apply \eqref{B3}; 
this gives
$$
\Ker (\bH_\bu^2-s^2I)=
U_\mu(p\Ran H_\theta)=U_\mu(p\calK_{z\theta})=\bp\bcalK_{\btheta}=\bp\Ran\bH_{\btheta}, 
$$
with 
$$
\bp=p\circ\mu, \quad 
\btheta(z)=\mu(z)\theta(\mu(z)). 
$$
Applying $U_\mu$ to \eqref{B2}, we obtain
$$
\bH_\bu U_\mu T_p=sU_\mu T_p H_\theta\text{ on }\Ran H_\theta.
$$
Observing that $U_\mu T_p=\bT_{\bp}U_\mu$, we arrive at
$$
\bH_\bu \bT_\bp U_\mu=s\bT_\bp U_\mu H_\theta\text{ on }\Ran H_\theta.
$$
By Proposition~\ref{prp.B3}, we have $U_\mu H_\theta=U_\mu K_{S\theta}=\bH_\btheta U_\mu$, 
and so 
$$
\bH_\bu \bT_\bp U_\mu=s\bT_\bp \bH_\btheta U_\mu \text{ on }\Ran H_\theta.
$$
Finally, by \eqref{B3} we have $U_\mu\Ran H_\theta=U_\mu\calK_{S\theta}=\Ran\bH_\btheta$. 
This completes the proof of part (i).

Part (ii) is a direct consequence of Theorem~\ref{thm.z0}(ii) and the fact that the degree of the inner 
factor of $p$ is invariant under the composition with $\mu$. 
\end{proof}

\appendix

\section{Linear Hankel operators}\label{sec.app}
Here we rewrite Theorem~\ref{thm.z0}(i) in terms of a representation for the Hankel matrix $\Gamma$ as a linear (rather than anti-linear) operator on the Hardy space. Let $J$ be the linear involution in $L^2(\bbT)$, 
$$
Jf(z)=f(\overline{z}), \quad z\in \bbT,
$$
and let $\bC$ be the anti-linear involution in $\calH^2(\bbT)$, 
$$
\bC f(z)=\overline{f(\overline{z})}, \quad z\in \bbT. 
$$
For a symbol $u\in \BMOA(\bbT)$, let us define the \emph{linear} Hankel operator $G_u$ in $\calH^2(\bbT)$ by 
$$
G_u f=P(u\cdot Jf). 
$$
It is straightforward to see that 
$$
(G_uz^n,z^m)=\wh u(n+m), 
$$
and so $G_u$ is unitarily equivalent to the Hankel matrix 
$\Gamma=\{\wh u(n+m)\}_{n,m=0}^\infty$. 
\begin{theorem}
Let $s$ be a singular value of $G_u$. Then there exists an inner function $\theta$ and an isometric multiplier $p$ on $\Ran H_\theta$ such that 
\begin{align*}
\Ker (G_u^*G_u-s^2I)&=\bC(p\Ran H_\theta). 
\\
\Ker (G_uG_u^*-s^2I)&=p\Ran H_\theta. 
\end{align*}
The action 
$$
G_u:\Ker (G_u^*G_u-s^2I)\to \Ker (G_uG_u^*-s^2I)
$$
is given by 
$$
G_u\bC(pf)=sp\theta\overline{f}, \quad
f\in \Ran H_\theta.
$$
\end{theorem}
This theorem immediately follows from Theorem~\ref{thm.z0} after identification 
$$
G_u=H_u\bC, \quad G_u^*=\bC H_u. 
$$


\end{document}